\documentclass[12pt,reqno]{amsart}

\usepackage{amsfonts,amssymb}
\usepackage{textcomp}
\usepackage{graphicx}
\usepackage{epsfig}
\usepackage{latexsym}
\usepackage{amsmath,amsthm}
\usepackage{mathrsfs}
\usepackage[all,cmtip]{xy}
\usepackage[pagebackref=true]{hyperref}
\hypersetup{   pdfborder={0 0 0},   
              colorlinks = false}
\usepackage{tikz-cd}
\usepackage{xcolor}
\usepackage{enumitem}

\renewcommand*{\backref}[1]{}
\renewcommand*{\backrefalt}[4]{%
  \ifcase #1 %
    No citations.
  \or
    ($\uparrow$ #4).%
  \else
    ($\uparrow$ #4).%
  \fi%
}
\theoremstyle{plain}
\newtheorem{theorem}{Theorem}[section]
\newtheorem{lemma}[theorem]{Lemma}
\newtheorem{definition}[theorem]{Definition}

\newtheorem{remark}[theorem]{Remark}
\newtheorem*{example}{Example}

\numberwithin{equation}{section}

\makeatletter
\newcommand*\bigcd{\mathpalette\bigcd@{1.4}}
\newcommand*\bigcd@[2]{\mathbin{\vcenter{\hbox{\scalebox{#2}{$\m@th#1\bullet$}}}}}
\makeatother
\makeatletter
\newcommand*\bigcdot{\mathpalette\bigcdot@{.8}}
\newcommand*\bigcdot@[2]{\mathbin{\vcenter{\hbox{\scalebox{#2}{$\m@th#1\bullet$}}}}}
\makeatother
\makeatletter
\newcommand*\bigcdo{\mathpalette\bigcdo@{.4}}
\newcommand*\bigcdo@[2]{\mathbin{\vcenter{\hbox{\scalebox{#2}{$\m@th#1\bullet$}}}}}
\makeatother

\makeatletter
\newcommand{\pushright}[1]{\ifmeasuring@#1\else\omit\hfill$\displaystyle#1$\fi\ignorespaces}
\newcommand{\pushleft}[1]{\ifmeasuring@#1\else\omit$\displaystyle#1$\hfill\fi\ignorespaces}
\makeatother

\newcommand{\bigslant}[2]{{\raisebox{.2em}{$#1$}\left/\raisebox{-.2em}{$#2$}\right.}}

\makeatletter
\def\moverlay{\mathpalette\mov@rlay}
\def\mov@rlay#1#2{\leavevmode\vtop{%
   \baselineskip\z@skip \lineskiplimit-\maxdimen
   \ialign{\hfil$\m@th#1##$\hfil\cr#2\crcr}}}
\newcommand{\charfusion}[3][\mathord]{
    #1{\ifx#1\mathop\vphantom{#2}\fi
        \mathpalette\mov@rlay{#2\cr#3}
      }
    \ifx#1\mathop\expandafter\displaylimits\fi}
\makeatother

\newcommand{\bigcupdot}{\charfusion[\mathop]{\bigcup}{\cdot}}

\newcommand{\la}[1]{\mathbb{#1}}

\hypersetup{colorlinks, citecolor=blue, filecolor=black, linkcolor=red}
\begin{document}
\baselineskip=15.5pt

\title{Moduli spaces of  Hom-Lie algebroid connections}

\author{Ayush~Jaiswal}
\address{Department of Mathematics, IISER Tirupati,
 Srinivasapuram-Jangalapalli Village, Panguru (G.P) Yerpedu Mandal, Tirupati - 517619, Chitoor District, Andhra Pradesh India.}
\email{ayushjwl.math@gmail.com; ayushjaiswal@labs.iisertirupati.ac.in}

\thanks{AJ is supported by the Institute Post-Doctoral Fellowship (IISER-T/Offer/PDRF(M)/A.J/04/2023) by IISER Tirupati.}
\date{}

\begin{abstract}
We have studied irreducible Hom-Lie algebroid connections for Hom-bundle and prove that the H-gauge theoretic moduli space has a Hausdorff Hilbert manifold structure. This work generalizes some known results about simple semi-connections and Lie algebroid connections for complex vector bundles on compact complex manifold.
\end{abstract}
\maketitle
\section{Introduction}
In \cite{JTHDLSDS}, Hartwig, Larson and Silverstov have introduced Hom-Lie algebra, while studying deformations of Witt and Virasoro algebras . A Hom-Lie algebra is a generalization of a Lie algebra. Similar to the case of Lie algebra, there is a Hom-Jacobi identity, which is also known as Jacobi identity, twisted by a linear map. There is a growing interest in Hom-Lie algebra as well as Hom-lie algebroids; some articles in this direction are \cite{LGCJT,LCJLYS,BSFS}.

Camille Laurent-Gengouxa and Joana Teles \cite{LGCJT}  introduced the notions of Hom-Lie algebroid, Hom-Gerstenhaber algebra, and Hom-Poisson structure. In \cite{LCJLYS}, Liqiang Cai, Jiefeng Liu, and Yunhe Sheng modified the definition of Hom-Lie algebroid and gave its dual description. In section \ref{sec-2}, we have described space of (irreducible) Hom-Lie algebroid connections $\big(\widehat{A}(E,\la{L})\big)$ $A(E,\la{L})$ and its affine space structure for a given Hom-bundle $E$ and a Hom-Lie algebroid $\la{L}$ on a smooth manifold $X$, following the definition of Cai, Liu and Sheng.

Gauge theory is useful in describing differential geometric invariants of manifolds like Seiberg-Witten and Donaldson invariants (see \cite{SKDPBK}), which are solutions of some gauge invariant equations.

In Section \ref{sec-3}, we have described Hom-gauge group $\text{H-Gau}(E)$ which has a Lie group structure, associated moduli space of irreducible Hom-Lie algebroid connections 
$$
\widehat{B}(E,\la{L})=\bigslant{\widehat{A}(E,\la{L})}{\text{H-Gau}(E)}$$ 
and their Sobolev completions in Section \ref{sec-4}. Furthermore, in Section \ref{sec-4}, we have described Hilbert manifold structure on the Sobolev space $\widehat{B}(E,\la{L})_l$ and proved 
$$
\widehat{p}:\widehat{A}(E,\la{L})_l\rightarrow \widehat{B}(E,\la{L})_l
$$
is a principal $\text{H-Gau}(E)^r_{l+1}$-bundle. Libor Krizka has proved similar results for Lie algebroid connections \cite{LK}.


\section{Preliminaries}\label{sec-2}
This section discusses the basic definitions of Hom-Lie algebra, Hom-Lie algebroid, and associated connections. Also, we verify the space of Hom-Lie algebroid connections is an affine space. We will use the notation $\mathbb{K}$ for the field $\mathbb{C}$ as well as $\mathbb{R}$. For more details about Hom-Lie algebras and Hom-Lie algebroids, see \cite{LCJLYS}. 
\subsection{Hom-Lie-algeba}
For a given smooth manifold $X$, with a smooth morphism $\phi:X\rightarrow X$, there is a canonical morphism (pullback) of the ring of smooth functions
$$
\phi^\star:C^\infty(X)\rightarrow C^\infty(X)
$$
that is $\phi^\star(fg)=\phi^\star(f)\phi^\star(g)$ for all $f,g\in C^\infty(X)$.

Let $E\rightarrow X$ be a smooth complex (resp. real) vector bundle over $X$ and $\phi:X\rightarrow X$ be a smooth map. The pullback bundle of $E$ along $\phi$, which we will denoted by $\phi^!E$ and for any $s\in \Gamma(X,E)$, we use $s^!\in\Gamma(X,\phi^!E)$ to denote the corresponding pullback section, i.e. $s^!(x)=\phi^\star(s)(x)=s\big(\phi(x)\big)$ for $x\in M$.
\begin{definition}\rm{
    For a smooth manifold $X$, an algebra morphism $D:C^\infty(X)\rightarrow C^\infty(X)$ is called $(\sigma,\tau)$-derivation, if 
    $$
    D(fg)=\sigma(f)D(g)+\tau(g)D(f)
    $$
for $f,g\in C^\infty(X);\sigma, \tau:C^\infty(X)\rightarrow C^\infty(X)$ ring morphisms.
    }
\end{definition}
The collection of $(\phi^\star,\phi^\star)$-derivations for the commutative $\mathbb{K}$-algebra $C^\infty(X)$ will be denoted by $\text{Der}_{\phi^\star,\phi^\star}\big(C^\infty(X)\big)$. For a given $D\in \text{Der}_{\phi^\star,\phi^\star}\big(C^\infty(X)\big)$, $\phi^\star\circ D\circ (\phi^\star)^{-1}\in \text{Der}_{\phi^\star,\phi^\star}\big(C^\infty(X)\big)$ because
\begin{align*}
\phi^\star\circ D\circ (\phi^\star)^{-1}(fg)&=\phi^\star\circ D\big((\phi^\star)^{-1}(f)(\phi^\star)^{-1}(g)\big)\\
&=\phi^\star\Big(f\big(D\circ (\phi^\star)^{-1}\big)(g)+g\big(D\circ(\phi^\star)^{-1}\big)(f)\Big)\\
&=\phi^\star(f)\big(\phi^\star\circ D\circ (\phi^\star)^{-1}\big)(g)+\phi^\star(g)\big(\phi^\star\circ D\circ(\phi^\star)^{-1}\big)(f)
\end{align*}
We have a canonical map
$$
\text{Ad}_{\phi^\star}:\text{Der}_{\phi^\star,\phi^\star}\big(C^\infty(X)\big)\rightarrow \text{Der}_{\phi^\star,\phi^\star}\big(C^\infty(X)\big)\quad\big(D\mapsto \phi^\star\circ D\circ (\phi^\star)^{-1}\big)
$$
\begin{remark}\label{pullbackact}\rm{
A pulled back section $D^!\in \Gamma(X,\phi^!T_X)$ acts on the ring of smooth functions given by
\begin{align*}
D^!(gf)=\big(D(gf)\big)^!=\big(gD(f)+fD(g)\big)^!&=g^!\big(D(f)\big)^!+f^!\big(D(g)\big)^!\\
&=g^!D^!(f)+f^!D^!(g)
\end{align*}
and can be considered as a $(\phi^\star,\phi^\star)$-derivation.
}
\end{remark}

\begin{definition}\rm{
A tuple $(V,[\cdot,\cdot]_V,\psi_V)$ containing a $\mathbb{K}$-vector space $V$, a skew-symmetric bilinear map $[\cdot, \cdot]_V: V \times V \rightarrow V$ and a morphism $\psi_V : V \rightarrow V$ satisfying the following properties  
\begin{enumerate}
\item $\psi_V([x,y]_V)=[\psi_V(x),\psi_V(y)]_V$, 
\item $[\psi_V(x),[y,z]_V]_V+[\psi_V(y),[z,x]_V]_V+ [\psi_V(z),[x,y]_V]_V=0$
\end{enumerate}
for all $x,y, z\in V$, is called a {\bf Hom-Lie algebra}.

	}
\end{definition}

The tuple $\big(V=\text{Der}_{\phi^\star,\phi^\star}\big(C^\infty(X)\big),[\cdot,\cdot]_V,\psi_V=\text{Ad}_{\phi^\star}\big)$ is a Hom-Lie algebra with the Hom-Lie bracket, given by
$$
[D_1,D_2]_V=\phi^\star\circ D_1\circ (\phi^\star)^{-1} \circ D_2 \circ (\phi^\star)^{-1}-\phi^\star\circ D_2\circ (\phi^\star)^{-1} \circ D_1 \circ (\phi^\star)^{-1}
$$
for $D_i\in \text{Der}_{\phi^\star,\phi^\star}\big(C^\infty(X)\big)$ $(i=1,2)$.
\begin{definition}\rm{
A representation of a Hom-Lie algebra $(V,[\cdot,\cdot]_V,\psi_V)$ on a vector space $V'$ with respect to $\beta \in \text{End}(V')$ is a linear map
		$\rho:V\rightarrow \text{End}(V')$, such that for all
		$x,y\in V$, the following equalities are satisfied:
        \begin{eqnarray*}
			\rho\big(\psi(x)\big)\circ \beta&=&\beta\circ \rho(x);\\
			\rho([x,y]_V)\circ
			\beta&=&\rho\big(\psi(x)\big)\circ\rho(y)-\rho\big(\psi(y)\big)\circ\rho(x).
		\end{eqnarray*}
	A representation of a Hom-Lie algebra $(V,[\cdot,\cdot]_V,\psi_V)$ on a vector space $V'$ w.r.t $\beta\in \text{End}(V')$ is denoted by  $(V',\beta,\rho)$.
}
\end{definition}
\subsection{Hom-Lie algebroid} A given vector bundle $E$ on a smooth manifold $X$, is said to have a Hom-bundle structure if there is an algebra morphism
$$
\phi_E:\Gamma(X,E)\rightarrow \Gamma(X,E)
$$
such that $\phi_E(fs)=\phi^\star(f)\phi_E(s)$ for all $f\in C^\infty(X)$ and $s\in \Gamma(X,E)$.

A Hom-bundle is said to be an invertible Hom-bundle, if the morphism $\phi$ is diffeomorphism and $\phi_E$ is an invertible linear map. From now on, we will assume a Hom-bundle is invertible.

For given two Hom-bundle $E=(E,\phi,\phi_E)$ and $F=(F,\phi,\phi_F)$, a vector bundle morphism $\psi:E\rightarrow F$ with same base is said to respect the Hom-structure, if 
$$
\psi\circ {\phi_E}={\phi_F}\circ \psi.
$$
Such a vector bundle morphism is called Hom-bundle morphism.

\begin{example} \rm{
Some examples of Hom-bundles are following.
\begin{enumerate}
\item For  a given Hom-bundle $E=(E,\phi,\phi_E)$ with an open subset $U\subset X$, satisfying $\phi(U)\subset U$, there is a Hom-bundle $(E_U,\phi_U,\phi_{E_U})$, where
$$
\phi_{E_U}:\Gamma(U,E)\rightarrow \Gamma(U,E) 
$$
is the restriction map ${\phi_E}_{\mkern 1mu \vrule height 2ex\mkern2mu \Gamma(U,E)}$. The Hom-bundle $(E_U,\phi_U,\phi_{E_U})$ is called Hom-subbundle restricted to $U$.
\item For a Hom-bundle $E=(E,\phi,\phi_E)$, there is a natural bundle $\text{End}(E)\cong \phi^!(E)\otimes \text{Hom}\big(E,\phi^!(\underline{\mathbb{K}})\big)$, where $\underline{\mathbb{K}}$ is trivial line bundle with fiber $\mathbb{K}$. The space of global sections is
$$
\Gamma\big(X,\text{End}(E)\big)=\{\psi:\Gamma(X,E)\rightarrow \Gamma(X,E)\mid \psi(fs)=\phi^\star(f)\psi(s)\}
$$
with a canonical morphism $\text{Ad}_{\phi_E}:\Gamma\big(X,\text{End}(E)\big)\rightarrow\Gamma\big(X,\text{End}(E)\big)$, given by
$$
\text{Ad}_{\phi_E}(\psi)=\phi_E\circ \psi\circ \phi_E^{-1}
$$ 
The tuple $(\text{End}(E),\phi,\text{Ad}_{\phi_E})$ is a Hom-bundle. The collection of global $\phi^\star$-linear morphisms respecting the Hom-bundle structure is given by
$$
\Gamma\big(X,\text{End}_{\phi_E}(E)\big)=\{\psi\in \Gamma(X,\text{End}(E)\mid \psi\circ\phi_E=\phi_E\circ \psi\}
$$
\item For a smooth manifold $X$ with smooth morphism $\phi:X\rightarrow X$, there is a smooth map $Ad_{\phi^\star}:\Gamma(X,\phi^!T_X)\rightarrow \Gamma(X,\phi^!T_X)$ given by,
$$
Ad_{\phi^\star}(D)=\phi^\star\circ D\circ (\phi^\star)^{-1}
$$
and the tuple $(\phi^!T_X,\phi,Ad_{\phi^\star})$ is a Hom-bundle.
\item For a given Hom-bundle $E=(E,\phi,\phi_E)$, there is a canonical map 
$$
\Gamma(X,\Lambda^pE)\rightarrow \Gamma(X,\Lambda^pE)\qquad\big(\wedge_{i=1}^pD_i\mapsto \wedge_{i=1}^p\phi_E(D_i)\big)\text{ for any }p\in \mathbb{Z}^+
$$
The induced map will also be denoted by the same symbol $\phi_E:\Lambda^pE\rightarrow \Lambda^pE$ and the tuple $(\Lambda^pE,\phi,\phi_E)$ is a Hom-bundle.
\item For a Hom-bundle $E=(E,\phi,\phi_E)$, there is a canonical map $\phi^\dagger_E:\Gamma(X,\Lambda^p E^\star) \rightarrow \Gamma(X,\Lambda^p E^\star)$ given by,
$$
\phi^\dagger_E(\omega)(X)=\phi^\star\Big(\omega\big(\phi_E^{-1}(X)\big)\Big)\quad\big(X\in \Gamma(X,\Lambda^p E)\big)
$$
The tuple $(\Lambda^p E^\star,\phi,\phi^\dagger_E)$ is a Hom-bundle.

%
\end{enumerate}
}
\end{example}

\begin{definition} \label{Hom-Lie algebroid}\rm{
A {\bf Hom-Lie algebroid} structure on a Hom-bundle $(E \rightarrow X,\phi,\phi_E)$ is a pair that consists of a Hom-Lie algebra structure $(\Gamma(X,E),[\cdot,\cdot]_{E},\phi_E)$ on the space of global sections, $\Gamma(X,E)$ and a bundle map $\mathfrak{a}_E: E \rightarrow \phi^!T_X$, called the anchor map, such that the following conditions are satisfied:
		\begin{enumerate}
		  \item For all $x, y\in \Gamma(X,E)$ and $f\in C^{\infty}(X)$, 
		  $$
		  [x,fy]_E=\phi^*(f)[x,y]_E+ \left[\mathfrak{a}_{E,X}\big(\phi_E (x)\big)(f)\right]\phi_E(y)
		  $$ 
		  where $\mathfrak{a}_{E,X}: \Gamma(X,E) \rightarrow \Gamma(X,\phi^!T_X)$ is the induced map and $\mathfrak{a}_{E,X}\big(\phi_E (x)\big)\in \Gamma(X,\phi^!T_X)$ acts on $f$ as in Remark \ref{pullbackact};
			\item $(C^{\infty}(X), \phi^*, \mathfrak{a}_{E,X})$  is a representation of the Hom-Lie algebra $(\Gamma(X,E),[\cdot,\cdot]_E,\phi_E)$ on vector space $C^\infty(X)$ w.r.t $\phi^\star\in \text{End}\big(C^\infty(X)\big)$.
		\end{enumerate}
}
\end{definition}	
A Hom-Lie algebroid is denoted by a quintuple $(E,\phi,\phi_E,[\cdot,\cdot]_E,\mathfrak{a}_E)$.
\begin{example}\rm{
Some examples of Hom-Lie algebroids are following.
\begin{enumerate}
\item For a given manifold $X$, there is a canonical Hom-bundle $(\phi^!T_X,\phi,\text{Ad}_{\phi^\star})$. Define a skew-symmetric bilinear map $[\cdot,\cdot]_{\phi^!T_X}:\phi^!TX\times \phi^!TX\rightarrow \phi^!TX$ given by
$$
[D_1,D_2]_{\phi^!TX}=\phi^\star\circ D_1\circ (\phi^\star)^{-1}\circ D_2\circ (\phi^\star)^{-1}-\phi^\star\circ D_2\circ (\phi^\star)^{-1}\circ D_1\circ (\phi^\star)^{-1}
$$
The tuple $(\phi^!TX,\phi,\text{Ad}_{\phi^\star},[\cdot,\cdot]_{\phi^!T_X},\text{id}_{\phi!T_X})$ is a Hom-Lie algebroid.
\item According to above example, pullback of tangent bundle $T_X$ along a smooth morphism $\phi:X\rightarrow X$ has a Hom-Lie algebroid structure. In general, for a given Lie algebroid $(\mathcal{A},[\cdot,\cdot]_\mathcal{A},\mathfrak{a}_\mathcal{A})$ with additional data there is an induced Hom-Lie algebroid. Let $(\mathcal{A},\phi,\alpha)$ be a Hom-bundle such that
\begin{enumerate}
\item $\alpha[x,y]_\mathcal{A}=[\alpha(x),\alpha(y)]_\mathcal{A}$
\item $\mathfrak{a}_\mathcal{A}\big(\alpha(x)\big)\circ \phi^\star(f)=\phi^\star\circ \mathfrak{a}_\mathcal{A}(x)(f)$
\end{enumerate}
then there is a canonical Hom-Lie algebroid $(A=\phi^!\mathcal{A},\phi,\phi_A=\alpha^!,[\cdot,\cdot]_A,\mathfrak{a}_A)$, where the maps $\phi_A:\Gamma(X,A)\rightarrow \Gamma(X,A)$, $[\cdot,\cdot]_A:\Gamma(X,A)\times \Gamma(X,A)\rightarrow \Gamma(X,A)$ and $\mathfrak{a}_A:A\rightarrow \phi!T_X$ are given by
$$
\phi_A(x^!)=\alpha(x)^!,\quad [x^!,y^!]_A=[x,y]_\mathcal{A}^!\text{ and }\quad \mathfrak{a}_A(x^!)=\mathfrak{a}_\mathcal{A}(x)^!
$$
\end{enumerate}
}
\end{example}
From now on, a Hom-Lie algebroid $(\la{L},\phi,\phi_\la{L},[\cdot,\cdot]_\la{L}, \mathfrak{a}_\la{L})$ will be denoted by $\la{L}$ , a Hom-bundle $(E,\phi,\phi_E)$ will be denoted by $E$ and the induced map $\mathfrak{a}_{\la{L},U}$ for any open subset $U\subset X$ will be denoted by $\mathfrak{a}_\la{L}$.


Using the dual anchor map $\mathfrak{a}_\la{L}^\star:\Gamma(X,\phi^!T_X^\star)\rightarrow \Gamma(X,\la{L}^\star)$, define a degree-1 differential operator on the ring of smooth functions 
$$
d_\la{L}=\mathfrak{a}_\la{L}^\star\circ d:C^\infty(X)\rightarrow \Gamma(X,\la{L}^\star)
$$ 
where $d:C^\infty(X)\rightarrow \Gamma(X,\phi^!T_X^\star)$ is the canonical operator induced from the de Rham operator.
The differential operator $d_\la{L}$ can be extended to higher exterior powers of $\la{L}^\star$ 
$$
d_\la{L}^p:\Gamma(X,\Lambda^p\la{L}^\star)\rightarrow \Gamma(X,\Lambda^{p+1}\la{L}^\star)
$$ 
given by
\begin{align*}
d_\la{L}^p(\omega)(\xi_0,\xi_1,\dots,\xi_p)&=\displaystyle\sum_{i=0}^p(-1)^i\mathfrak{a}_\la{L}(\xi_i)\Big(\omega\big({\phi_\la{L}}^{-1}(\xi_0),{\phi_\la{L}}^{-1}(\xi_1),\dots,\widehat{\xi_i},\dots,{\phi_\la{L}}^{-1}(\xi_p)\big)\Big)\\
&\quad +\displaystyle\sum_{i<j,i=0}^{p}(-1)^{i+j}\phi^\dagger_\la{L}(\omega)\big([{\phi_\la{L}}^{-1}(\xi_i),{\phi_\la{L}}^{-1}(\xi_j)]_\la{L},\xi_0,\xi_1,\dots,\xi_p\big)\qquad\big(\xi_i\in \Gamma(X,\la{L})\big)
\end{align*}
For each $p\geq 0$ with $d_\la{L}^0=d_\la{L}$, it is easy to observe that $d_\la{L}^{p}\circ d_\la{L}^{p-1}=0$ (see \cite[Theorem 3.7]{LCJLYS}). The graded differential complex $(\Lambda^\bullet \la{L}^\star,d_\la{L}^\bullet)$ is called {\bf Hom-Chevalley-Eilenberg-de Rham} complex or $(\phi^\dagger_\la{L},\phi^\dagger_\la{L})$-differential graded complex.

Using abuse of notation, the operator $d_\la{L}^\bullet:\Lambda^\bullet \la{L}^\star\rightarrow \Lambda^{\bullet+1}\la{L}^\star$ will be denoted by $d_\la{L}$, is called {\bf Hom-Chevalley-Eilenberg-de Rham operator}, which is a generalization of the de Rham as well as the Chevalley-Eilenberg-de Rham operator.

For $\xi,\xi_i\in \Gamma(X,\la{L})$ $(i\in I)$, the Hom-Lie derivative operator $\mathcal{L}_\xi^\la{L}:\Lambda^p \la{L}^\star\rightarrow \Lambda^p \la{L}^\star$ is given by
\begin{align*}
(\mathcal{L}_\xi^\la{L}\omega)(\xi_1,\xi_2,\dots,\xi_p)&=\mathfrak{a}_\la{L}\big(\phi_\la{L}(\xi)\big)\Big(\omega\big({\phi_\la{L}}^{-1}(\xi_1),{\phi_\la{L}}^{-1}(\xi_2),\dots,{\phi_\la{L}}^{-1}(\xi_p)\big)\Big)\\
& \qquad -\sum_{i=1}^p(\phi^\dagger_\la{L}\omega)(\xi_1,\dots,[\xi,{\phi_\la{L}}^{-1}(\xi_i)]_\la{L},\dots,\xi_p)
\end{align*}
and the Hom-Insertion operator $i_\xi^\la{L}:\Gamma(X,\Lambda^p\la{L}^\star)\rightarrow \Gamma(X,\Lambda^{p-1}\la{L}^\star)$ is given by 
\begin{align*}
(i_\xi^\la{L}\omega)(\xi_1,\xi_2,\dots,\xi_{p-1})&=\phi^\dagger_\la{L}\omega\big(\phi_\la{L}(\xi),\xi_1,\xi_2,\dots,\xi_p\big)
\end{align*}
\begin{remark}\rm{
For a smooth Hom-Lie algebroid $\la{L}$ and a smooth Hom-bundle $E$ on a smooth manifold $X$, the space of global sections of the bundle $\Lambda^k\la{L}^\star$ (resp. $\Lambda^k\la{L}^\star\otimes E$) will be denoted by $\mathcal{A}^k_{\la{L}}(X)$ (resp. $\mathcal{A}^k_{\la{L}}(X,E)$)$(0\leq k\in \mathbb{Z})$, whose elements are called smooth $\la{L}$-forms (resp. $\la{L}$-forms taking values in $E$) of type $k$.
}
\end{remark}
\begin{remark}\rm{
For a given Hom-bundle $(E,\phi,\phi_E)$ and a Hom-Lie algebroid $\la{L}$, the bundle $\Lambda^k\la{L}^\star\otimes \text{End}(E)$ has induced Hom structure with the algebra morphism
\begin{align*}
\phi_\la{L}^\dagger\otimes\text{Ad}_{\phi_E}:\Gamma\big(X,\Lambda^k\la{L}^\star\otimes \text{End}(E)\big)&\rightarrow \Gamma\big(X,\Lambda^k\la{L}^\star\otimes \text{End}(E)\big)\qquad \text{given by,}\\
\phi_\la{L}^\dagger\otimes\text{Ad}_{\phi_E}(D)&=(\phi^\dagger_\la{L}\otimes \phi_E)\circ D\circ \big(\phi_E^{-1}\big)
\end{align*}
The global sections of bundle $\Lambda^k\la{L}^\star\otimes \text{End}(E)$ respecting the Hom-structure is given by
$$
\mathcal{A}^k_\la{L}\big(X,\text{End}_{\phi_E}(E)\big)=\{ D\in \mathcal{A}^k_\la{L}\big(X,\text{End}(E)\big)\mid \phi^\dagger_\la{L}\otimes \phi_E\circ D=D\circ \phi_E\}
$$
}
\end{remark}
\subsection{Hom-Lie algebroid connection}
\begin{definition}\rm{}
Let $\la{L},E$ be a Hom-Lie algebroid and a Hom-bundle, respectively over a smooth manifold $X$. A $\mathbb{K}$-linear map $\nabla: \mathcal{A}^0_{\la{L}}(X,E)  \rightarrow \mathcal{A}^1_{\la{L}}(X,E)$
such that for $f\in C^\infty(X);s\in \Gamma(X,E)$
\begin{enumerate}
\item $\nabla(fs)=d_\la{L}(f)\phi_E(s)+\phi^\star(f)\nabla(s)$
\item $\phi^\dagger_\la{L}\otimes \phi_E\circ \nabla=\nabla\circ \phi_E$
\end{enumerate}
 is called a Hom-Lie algebroid connection or, an $\la{L}$-connection.
\end{definition}
A Hom-Lie algebroid connection is a generalization of a Lie algebroid as well as an affine connection. For a given $\xi\in \Gamma(X,\la{L})$, we have a linear morphism $\nabla_\xi\in \mathcal{A}^0_\la{L}\big(X,\text{End}_{\phi_E}(E)\big)$ such that  
\begin{align*}
\nabla_\xi(fs)&=\mathcal{L}^\la{L}_\xi(f)\phi_E(s)+\nabla_\xi(s)\phi^\star(f)\\
\nabla_{f\xi}(s)&=\phi^\star(f)\nabla_\xi(s)
\end{align*} 
for $f\in \mathcal{C}^\infty(X);s\in \Gamma(X,E)$.

The element $\nabla_\xi(fs)$ is called co-variant $\la{L}$-derivative and the morphism $\nabla_\xi$ is called co-variant $\la{L}$-differential operator in the direction of $\xi\in \Gamma(X,\la{L})$.

%
\begin{theorem}\label{operconn}
For a given $\la{L}$-connection $\nabla$, there is a degree 1 differential operator $d^\nabla:\mathcal{A}^{\bigcdot}_\la{L}(X,E)\rightarrow \mathcal{A}^{\bigcdot}_\la{L}(X,E)$, uniquely determined such that
\begin{enumerate}
\item for $\alpha\in \mathcal{A}^l_\la{L}(X),\beta\in \mathcal{A}^\bullet_\la{L}(X,E)$
\begin{align}
d^\nabla(\alpha\wedge \beta)&=d_\la{L}(\alpha)\wedge (\phi^\dagger_\la{L}\otimes \phi_E)(\beta)+(-1)^l\phi^\dagger_\la{L}(\alpha)\wedge d^\nabla(\beta)\label{operconn1} &&&\\
\phi^\dagger_\la{L}\otimes \phi_E\circ d^\nabla&=d^\nabla\circ \phi^\dagger_\la{L}\otimes \phi_E\label{operconn2} &&&
\end{align}
\item for $s\in \mathcal{A}^0_\la{L}(X,E), \xi\in \Gamma(X,\la{L})$
\begin{align*}
d^\nabla_{\mkern 1mu \vrule height 2ex\mkern2mu {\mathcal{A}^0_\la{L}(X,E)}}\equiv \nabla\text{ and }\big(d^\nabla(s)\big)(\xi)&=\nabla_{\phi_\la{L}^{-1}(\xi)}(s) &&&&&&&&&
\end{align*}
\end{enumerate}
The operator $d^\nabla$ is given by 
\begin{align}\label{connope}
d^\nabla(\omega)(\xi_0,\xi_1,\dots,\xi_p)&=\displaystyle\sum_{i=1}^p(-1)^i\nabla_{\phi_\la{L}^{-1}(\xi_i)}\Big(\omega\big(\phi_\la{L}^{-1}(\xi_0),\phi_\la{L}^{-1}(\xi_1),\dots,\widehat{\xi_i},\dots,\phi_\la{L}^{-1}(\xi_p)\big)\Big)+\nonumber\\
&\qquad\displaystyle\sum_{i<j=0}^k (-1)^{i+j}\phi^\dagger_\la{L}\otimes \phi_E(\omega)\big([\phi_\la{L}^{-1}(\xi_i),\phi_\la{L}^{-1}(\xi_j)]_\la{L},\xi_0,\xi_1,\dots,\widehat{\xi_i},\dots,\widehat{\xi_j},\dots,\xi_p\big)
\end{align}
for $\xi_i(0\leq i\leq p)\in \Gamma(X,\la{L})$.
\end{theorem}
\begin{proof}
The uniqueness of the operator $d^\nabla$ satisfying properties (1) and (2) follows in the same line of arguements as in the case of affine connections only thing we need to prove is, the operator $d^\nabla$ described in the Equation \ref{connope} satisfies the properties (1) and (2). We have verified the property (1), verification of (2) is easy.
\begin{align*}
d^\nabla(f\beta)(\xi_0,\xi_1,\dots,\xi_p)&=\displaystyle\sum_{i=0}^p(-1)^i\nabla_{\phi_\la{L}^{-1}(\xi_i)}\Big(f\beta\big(\phi_\mathbb{L}^{-1}(\xi_0),\dots,\widehat{\xi}_i,\dots,\phi_\mathbb{L}^{-1}(\xi_p)\big)\Big)+\\
&\qquad\displaystyle\sum_{i<j}(-1)^{i+j}\phi_\la{L}^\dagger\otimes \phi_E(f\beta)\big([\phi_\la{L}^{-1}(\xi_i),\phi_\la{L}^{-1}(\xi_j)]_\la{L},\xi_0,\dots,\widehat{\xi}_i,\dots,\widehat{\xi}_j,\dots,\xi_p\big)\\
&=\displaystyle\sum_{i=0}^p(-1)^i\mathcal{L}^\la{L}_{\phi_\la{L}^{-1}(\xi_i)}(f)\Big(\phi^\dagger_\la{L}\otimes \phi_E(\beta)\big(\phi_\mathbb{L}^{-1}(\xi_0),\dots,\widehat{\xi}_i,\dots,\phi_\mathbb{L}^{-1}(\xi_p)\big)\Big)+\\
&\qquad \displaystyle\sum_{i=0}^p(-1)^i\phi^\star(f)\nabla_{\phi_\la{L}^{-1}(\xi_i)}\big(\beta(\phi_\mathbb{L}^{-1}(\xi_0),\dots,\widehat{\xi}_i,\dots,\phi_\mathbb{L}^{-1}(\xi_p)\big)+\\
&\qquad\qquad\displaystyle\sum_{i<j}(-1)^{i+j}\phi_\la{L}^\dagger\otimes \phi_E(f\beta)\big([\phi_\la{L}^{-1}(\xi_i),\phi_\la{L}^{-1}(\xi_j)]_\la{L},\xi_0,\dots,\widehat{\xi}_i,\dots,\widehat{\xi}_j,\dots,\xi_p\big)\\
&=\big(d_\la{L}(f)\wedge \phi^\dagger_\la{L}\otimes \phi_E(\beta)+\phi^\star(f)d^\nabla(\beta)\big)(\xi_0,\xi_1,\dots,\xi_p)
\end{align*}
Assume for $k\leq n$, the Equation \eqref{operconn1} holds, we need to prove for $k=n+1$. Let $\alpha\in \Gamma(X,\wedge^n\la{L}^\star),\beta\in \Gamma(X,\Lambda^n\la{L}^\star\otimes E)$ and $\xi\in \Gamma(X,\la{L}^\star)$,
{\allowdisplaybreaks
\begin{align*}
d^\nabla(\alpha\wedge \xi\wedge\beta)&=d_\la{L}(\alpha)\wedge \phi_\la{L}^\dagger\otimes \phi_E(\xi\wedge\beta)+(-1)^n\phi_\la{L}^\dagger(\alpha)\wedge d^\nabla(\xi\wedge\beta)\\
&=\quad d_\la{L}(\alpha)\wedge \Big(\phi_\la{L}^\dagger(\xi)\wedge\big(\phi_\la{L}^\dagger\otimes \phi_E(\beta)\big)\Big)+\\
&\qquad(-1)^n\phi_\la{L}^\dagger(\alpha)\wedge \left[d_\la{L}(\xi)\wedge \phi_\la{L}^\dagger\otimes \phi_E(\beta)-\phi_\la{L}^\dagger(\xi)\wedge d^\nabla(\beta)\right]\\
&=\quad \big(d_\la{L}(\alpha)\wedge \phi_\la{L}^\dagger(\xi)+(-1)^n\phi_\la{L}^\dagger(\alpha)\wedge d_\la{L}(\xi)\big)\wedge \big(\phi_\la{L}^\dagger\otimes \phi_E(\beta)\big)+\\
&\qquad(-1)^{n+1}\phi_\la{L}^\dagger(\alpha\wedge\xi)\wedge d^\nabla(\beta)\\
&=d_\la{L}(\alpha\wedge\xi)\wedge\phi_\la{L}^\dagger\otimes \phi_E(\beta)+(-1)^{n+1}\phi_\la{L}^\dagger(\alpha\wedge\xi)d^\nabla(\beta)
\end{align*}
}
Hence, for $n=k+1$, $d^\nabla$ described in the Equation \eqref{connope} satifies property (1). Now we will verify the Equation \eqref{operconn2}.
{\allowdisplaybreaks
\begin{align*}
d^\nabla\circ \phi_\la{L}^\dagger\otimes \phi_E(\omega)(\xi_0,\dots,\xi_p)&=\displaystyle\sum_{i=0}^p (-1)^i\nabla_{\phi_\la{L}^{-1}(\xi_i)}\Big[\big(\phi_\la{L}^\dagger\otimes \phi_E(\omega)\big)\big(\phi_\la{L}^{-1}(\xi_0),\dots,\widehat{\xi_i},\dots \phi_\la{L}^{-1}(\xi_p)\big)\Big]+\\
&\quad\displaystyle\sum_{i<j=0}^k (-1)^{i+j}\big((\phi^\dagger_\la{L}\otimes \phi_E)^2(\omega)\big)\big([\phi_\la{L}^{-1}(\xi_i),\phi_\la{L}^{-1}(\xi_j)]_\la{L},\xi_0,\dots,\widehat{\xi_i},\dots,\widehat{\xi_j},\dots,\xi_p\big)\\
&=\displaystyle\sum_{i=0}^p (-1)^i\bigg(\nabla\Big[\big(\phi_\la{L}^\dagger\otimes \phi_E(\omega)\big)\big(\phi_\la{L}^{-1}(\xi_0),\dots,\widehat{\xi_i},\dots \phi_\la{L}^{-1}(\xi_p)\big)\Big]\bigg)(\xi_i)+\\
&\quad\displaystyle\sum_{i<j=0}^k (-1)^{i+j}\big((\phi^\dagger_\la{L}\otimes \phi_E)^2(\omega)\big)\big([\phi_\la{L}^{-1}(\xi_i),\phi_\la{L}^{-1}(\xi_j)]_\la{L},\xi_0,\dots,\widehat{\xi_i},\dots,\widehat{\xi_j},\dots,\xi_p\big)\\
&=\displaystyle\sum_{i=0}^p (-1)^i\phi_E\circ\bigg(\nabla\left[\omega\big(\phi_\la{L}^{-2}(\xi_0),\dots,\widehat{\xi_i},\dots \phi_\la{L}^{-2}(\xi_p)\big)\right]\big(\phi_\la{L}^{-1}(\xi_i)\big)\bigg)+\\
&\quad\displaystyle\sum_{i<j=0}^k (-1)^{i+j}\phi_E\circ\Big(\big((\phi^\dagger_\la{L}\otimes \phi_E)(\omega)\big)\big([\phi_\la{L}^{-2}(\xi_i),\phi_\la{L}^{-2}(\xi_j)]_\la{L},\phi_\la{L}^{-1}(\xi_0),\dots,\widehat{\xi_i},\dots\\
&\qquad\qquad\qquad,\widehat{\xi_j},\dots,\phi_\la{L}^{-1}(\xi_p)\big)\Big)\\
&=\phi_E\circ \Big[d^\nabla(\omega)\big(\phi_\la{L}^{-1}(\xi_0),\dots,\phi_\la{L}^{-1}(\xi_p)\big)\Big]\\
&=\Big[\phi_\la{L}^\dagger\otimes \phi_E\circ d^\nabla(\omega)\Big]\big(\xi_0,\dots,\xi_p\big)
\end{align*}
}
\end{proof}
\begin{lemma}\label{liealgconn}
The space of $\la{L}$-connections $A(E,\la{L})$, is an affine space modeled on the vector space $\mathcal{A}^1_\la{L}\big(X,\text{End}_{\phi_E}(E)\big)$
\end{lemma}
\begin{proof}
The proof follows in two steps. First we prove, the space of $\la{L}$-connections $A(E,\la{L})$ is non-empty while in second step, we prove it is an affine space modelled on $\mathbb{K}$-vector space $\mathcal{A}^1_\la{L}\big(X,\text{End}_{\phi_E}(E)\big)$. 
\begin{enumerate}
\item Let $\{(U_\alpha,\psi_\alpha)\}_{\alpha\in I}$ be a smooth trivialization for the Hom-bundle $E$ with $\phi(U_\alpha)\subset U_\alpha$, take $(g_\alpha)_{\alpha\in \mathcal{A}}$ be a smooth partition of unity sub-ordinate to covering $(U_\alpha)_{\alpha\in I}$. 
Let $(\underline{V}=X\times V,\phi,\phi_{\underline{V}})$ be a trivial Hom-bundle, where ${\phi_{\underline{V}}}_{\mkern 1mu \vrule height 2ex\mkern2mu \Gamma(U_\alpha,\underline{V})}:\Gamma(U_\alpha,\underline{V})\rightarrow \Gamma(U_\alpha,\underline{V})$ is given by
\begin{align}\label{map1}
{\phi_{\underline{V}}}_{\mkern 1mu \vrule height 2ex\mkern2mu \Gamma(U_\alpha,\underline{V})}=\psi_\alpha\circ{\phi_E}_{\mkern 1mu \vrule height 2ex\mkern2mu \Gamma(U_\alpha,E)}\circ \psi_\alpha^{-1}
\end{align}
Let $\widehat{\nabla}$ be the trivial $\la{L}$-connection on trivial Hom-bundle $(\underline{V}, \phi, \phi_{\underline{V}})$ that is
$$
\widehat{\nabla}(f\otimes v)=d_{\la{L}}(f)\otimes \phi_{\underline{V}}(v)\quad \big(f\otimes v\in \Gamma(X,\underline{V})\big)
$$
For $s\in \Gamma(X,E)$, $\text{supp}(g_\alpha s)\subset U_\alpha$ implies $\text{supp}\Big(\psi_\alpha\big((g_\alpha s)_{\mkern 1mu \vrule height 2ex\mkern2mu U_\alpha}\big)\Big)\subset U_\alpha$, extending $\psi_\alpha\big((g_\alpha s)_{\mkern 1mu \vrule height 2ex\mkern2mu U_\alpha}\big)$ to $X$ and applying the trivial connection $\widehat{\nabla}$, we have 
$$\text{supp}\bigg(\widehat{\nabla}\Big(\psi_\alpha\big((g_\alpha s)_{\mkern 1mu \vrule height 2ex\mkern2mu U_\alpha}\big)\Big)\bigg)\subset U_\alpha.$$ Furthermore, $\text{supp}\Bigg(\text{id}_{\la{L}^\star}\otimes \psi_\alpha^{-1}\circ \bigg(\widehat{\nabla}\Big(\psi_\alpha\big((g_\alpha s)_{\mkern 1mu \vrule height 2ex\mkern2mu U_\alpha}\big)\Big)\bigg)_{\mkern 1mu \vrule height 2ex\mkern2mu U_\alpha}\Bigg)\subset U_\alpha$, which can be extended to a global section, $\text{id}_{\la{L}^\star}\otimes \psi_\alpha^{-1}\circ \bigg(\widehat{\nabla}\Big(\psi_\alpha\big((g_\alpha s)_{\mkern 1mu \vrule height 2ex\mkern2mu U_\alpha}\big)\Big)\bigg)_{\mkern 1mu \vrule height 2ex\mkern2mu U_\alpha}\in \Gamma(X,\la{L}^\star\otimes E)$. Define, a natural map
$\nabla:\Gamma(X,E)\rightarrow \Gamma(X,\la{L}^\star\otimes E)$ given by
$$
\nabla(s)=\sum_{\alpha\in I} \text{id}_{\la{L}^\star}\otimes \psi_\alpha^{-1}\circ \bigg(\widehat{\nabla}\Big(\psi_\alpha\big((g_\alpha s)_{\mkern 1mu \vrule height 2ex\mkern2mu U_\alpha}\big)\Big)\bigg)_{\mkern 1mu \vrule height 2ex\mkern2mu U_\alpha}
$$
The map $\nabla$ is an $\la{L}$-connection because for $f\in C^\infty(X), s\in \Gamma(X,E)$, we have
{\allowdisplaybreaks
\begin{align*}
\nabla(fs)&=\sum_{\alpha\in I} \text{id}_{\la{L}^\star}\otimes \psi_\alpha^{-1}\circ \bigg(\widehat{\nabla}\Big(\psi_\alpha\big((g_\alpha fs)_{\mkern 1mu \vrule height 2ex\mkern2mu U_\alpha}\big)\Big)\bigg)_{\mkern 1mu \vrule height 2ex\mkern2mu U_\alpha}\\
&=\sum_{\alpha\in I} \text{id}_{\la{L}^\star}\otimes \psi_\alpha^{-1}\circ \bigg(\widehat{\nabla}\Big(f\psi_\alpha\big((g_\alpha s)_{\mkern 1mu \vrule height 2ex\mkern2mu U_\alpha}\big)\Big)\bigg)_{\mkern 1mu \vrule height 2ex\mkern2mu U_\alpha}\\
&=\sum_{\alpha\in I} \text{id}_{\la{L}^\star}\otimes \psi_\alpha^{-1}\circ \bigg(d_\la{L}(f)\otimes\phi_{\underline{V}}\Big(\psi_\alpha\big((g_\alpha s)_{\mkern 1mu \vrule height 2ex\mkern2mu U_\alpha}\big)\Big)+\phi^\star(f)\widehat{\nabla}\Big(\psi_\alpha\big((g_\alpha s)_{\mkern 1mu \vrule height 2ex\mkern2mu U_\alpha}\big)\Big)\bigg)_{\mkern 1mu \vrule height 2ex\mkern2mu U_\alpha}\\
&=\sum_{\alpha\in I}  \left[d_\la{L}(f)\otimes {\phi_E}_{\mkern 1mu \vrule height 2ex\mkern2mu \Gamma(U_\alpha,E)}\big((g_\alpha s)_{\mkern 1mu \vrule height 2ex\mkern2mu U_\alpha}\big)+\phi^\star(f)\Bigg(\text{id}_{\la{L}^\star}\otimes \psi_\alpha^{-1}\circ\bigg(\widehat{\nabla}\Big(\psi_\alpha\big((g_\alpha s)_{\mkern 1mu \vrule height 2ex\mkern2mu U_\alpha}\big)\Big)\bigg)_{\mkern 1mu \vrule height 2ex\mkern2mu U_\alpha}\Bigg)\right]\\
&=d_\la{L}(f)\otimes \phi_E(s)+\phi^\star(f)\nabla(s)
\end{align*}
}
Because trivial $\la{L}$-connection $\widehat{\nabla}$ satisfies $\phi^\dagger_\la{L}\otimes \phi_V\circ \widehat{\nabla}=\widehat{\nabla}\circ \phi_V$ and using the Equation \eqref{map1}, we have $\phi^\dagger_\la{L}\otimes \phi_E\circ \nabla=\nabla\circ \phi_E$. 
Hence, the space of $\la{L}$-connections, $A(E,\la{L})$ is non-empty.
\item Let $\nabla, \nabla'\in A(E,\la{L})$, then 
\begin{align*}
(\nabla-\nabla')(fs)&=\big(d_\la{L}(f)\phi_E(s)+\phi^\star(f)\nabla(s)\big)-\big(d_\la{L}(f)\phi_E(s)+\phi^\star(f)\nabla'(s)\big)\\
&=\phi^\star(f)(\nabla-\nabla')(s)
\end{align*}
\end{enumerate}
It proves, the space $A(E,\la{L})$ is an affine space modelled on the vector space $\mathcal{A}^1_\la{L}\big(X,\text{End}_{\phi_E}(E)\big)$.
\end{proof}
\begin{remark}\rm{
Using Lemma \ref{liealgconn}, any Hom-Lie algebroid connection $\nabla\in A(E,\la{L})$ can be expressed as $\nabla=\nabla_0+\alpha$ for some $\alpha\in \mathcal{A}^1_\la{L}\big(X,\text{End}_{\phi_E}(E)\big)$, where $\nabla_0:\mathcal{A}^0_\la{L}(X,E)\rightarrow \mathcal{A}^1_\la{L}(X,E)$ is the trivial $\la{L}$-connection. The space of $\la{L}$-connections is given by
\begin{align}
A(E,\la{L})(X)&=\left\{\nabla_0+\alpha\mid \alpha\in \mathcal{A}^1_\la{L}\big(X,\text{End}_{\phi_E}(E)\big)\right\}
\end{align}
}
\end{remark}
\begin{remark}\label{remar}\rm{
Any $\la{L}$-connection $\nabla^E$ on a Hom-bundle $E$ induces an $\la{L}$-connection $\nabla^{\text{End}(E)}$ for the bundle $\text{End}(E)$, given by
\begin{align*}
(\nabla^{\text{End}(E)}T)(s)&=\phi^\dagger_\la{L}\otimes\phi_E\circ\nabla^E\circ\phi_E^{-1}\circ T\circ\phi_E^{-1}(s)-\phi^\dagger_\la{L}\otimes \phi_E\circ\text{id}_{\la{L}^\star}\otimes(T\circ \phi_E^{-1})\circ \nabla^E\circ\phi_E^{-1}(s)\\
&=[\nabla^E,T]_\la{L}(s)
\end{align*}
for $T\in \mathcal{A}^0_\la{L}\big(X,\text{End}(E)\big);s\in \Gamma(X,E)$.
}
\end{remark}
\begin{remark}\label{remar1}\rm{
For $T_i\in \mathcal{A}^0_\la{L}\big(X,\text{End}_{\phi_E}(E)\big)$ $(i=1,2)$ and a Hom-Lie algebroid connection $\nabla^E$, we have
\begin{align*}
\big(\nabla^{\text{End}(E)}(T_1\circ \phi_E^{-1}\circ T_2)\big)&=[\nabla^E,T_1\circ \phi_E^{-1}\circ T_2]_\la{L}\\
&=\phi^\dagger_\la{L}\otimes \phi_E\circ[\nabla^E,T_1\circ \phi_E^{-1}\circ T_2]\circ \phi_E^{-1}\\
&=\phi^\dagger_\la{L}\otimes \phi_E\circ\big([\nabla^E,T_1]\circ \phi_E^{-1}\circ T_2+\text{id}_{\la{L}^\star}\otimes (T_1\circ \phi_E^{-1})\circ[\nabla^E,T_2]\big)\circ \phi_E^{-1}\\
&=[\nabla^E,T_1]_\la{L}\circ \phi_E^{-1}\circ T_2+\text{id}_{\la{L}^\star}\otimes (T_1\circ \phi_E^{-1})\circ[\nabla^E,T_2]_\la{L}
\end{align*}
}
\end{remark}

The elements of the vector space $\mathcal{A}^{\bullet}_\la{L}\big(X,\text{End}(E)\big)$ has a canonical associative algebra structure with wedge product as multiplication operation given by 
\begin{equation}\label{wedgehom}
\omega\wedge\tau(\xi_1,\xi_2,\dots,\xi_{p+q})=\frac{1}{p!q!}\displaystyle\sum_{\sigma\in S(p+q)} Sign(\sigma)\omega(\xi_{\sigma(1)},\dots,\xi_{\sigma(p)})\circ\phi_E^{-1}\circ\tau(\xi_{\sigma(p+1},\dots,\xi_{\sigma(p+q)})
\end{equation}
for $\omega\in \mathcal{A}^p_\la{L}\big(X,\text{End}(E)\big)$ and $\tau\in \mathcal{A}^q_\la{L}\big(X,\text{End}(E)\big)$, $\omega\wedge \tau\in \mathcal{A}^{p+q}_\la{L}\big(X,\text{End}(E)\big)$.

The space $\mathcal{A}^{\bullet}_\la{L}\big(X,\text{End}(E)\big)$ also has, Hom-Lie algebra structure with the Hom-Lie bracket $[\cdot,\cdot]_\la{L}$ can be described as follows.

For $\omega,\tau\in \mathcal{A}^0_\la{L}\big(X,\text{End}(E)\big)$
\begin{equation}\label{eq.1}
[\omega,\tau]_{\la{L}}=\phi_E\circ\omega\circ\phi_E^{-1}\circ \tau\circ \phi_E^{-1}-\phi_E\circ\tau\circ\phi_E^{-1}\circ \omega\circ \phi_E^{-1}
\end{equation}
and for $\Omega\in \mathcal{A}^p_\la{L}\big(X,\text{End}(E)\big),\mathrm{T}\in \mathcal{A}^q_\la{L}\big(X,\text{End}(E)\big)$ $(p,q>0)$
\begin{align}\label{liealgebra}
[\Omega,\mathrm{T}]_\la{L}(\xi_1,\dots,\xi_{p+q})&=\frac{1}{p! q!}\displaystyle\sum_{\sigma\in S(p+q)} Sign(\sigma)\Big[\Omega\big({\phi_E}^{-1}(\xi_{\sigma(1)}),\dots,{\phi_E}^{-1}(\xi_{\sigma(p)})\big),\nonumber\\&\qquad\qquad\qquad\qquad\qquad\qquad\mathrm{T}\big({\phi_E}^{-1}(\xi_{\sigma(p+1)}),\dots,{\phi_E}^{-1}(\xi_{\sigma_{(p+q)}})\big)\Big]_\la{L}
\end{align}
for $\xi_1,\dots,\xi_{p+q}\in \Gamma(X,\la{L})$. 
\section{H-Gauge group action}\label{sec-3}
In this section, we have described Hom-gauge group (H-gauge group) for a given Hom-bundle along with the moduli space of irreducible $\la{L}$-connections as H-gauge equivalence classes, which we have further studied in Section \ref{sec-4}.

For a given Hom-bundle $E$, the collection of invertible $\phi^\star$-linear global morphisms has a group structure with the group operation $\psi\bigcdot \psi_2=\psi_1\circ \phi_E^{-1}\circ \psi_2$ and identity $\phi_E$, which we will call the Hom-Gauge group, given by 
\begin{align*}
  \text{H-Gau}(E) &=\left\{ \psi\in \mathcal{A}^0_\la{L}\big(X,\text{End}_{\phi_E}(E)\big)\ \middle\vert \begin{array}{l}
    \psi\bigcdot \psi'=\psi'\bigcdot\psi=\phi_E \\
    \text{ for some }\psi'\in \mathcal{A}^0_\la{L}\big(X,\text{End}_{\phi_E}(E)\big)
  \end{array}\right\} 
\end{align*}
Note that there is a unique element $[\psi]^{-1}=\phi_E\circ \psi^{-1}\circ \phi_E\in\text{H-Gau}(E)$ for each $\psi\in \text{H-Gau}(E)$ such that 
$$
\psi\bigcdot [\psi]^{-1}=[\psi]^{-1}\bigcdot\psi=\phi_E
$$
For a given smooth Hom-bundle $(E,\phi,\phi_E)$ on a smooth manifold $X$, define a map 
\begin{align}\label{gaugeact}
\odot:\text{H-Gau}(E)\otimes A(E,\la{L})&\rightarrow A(E,\la{L})\text{ given by,}\nonumber\\
\nabla^\psi=\odot(\psi,\nabla)&=\text{id}_{\la{L}^\star}\otimes \big([\psi]^{-1}\circ\phi_E^{-1}\big) \circ\nabla\circ\big(\phi_E^{-1}\circ \psi\big)\nonumber\\
&=\text{id}_{\la{L}^\star}\otimes [\psi]^{-1} \bigcdot\nabla\bigcdot \psi
\end{align}
It is easy to observe that
\begin{enumerate}
\item for $f\in C^\infty(X);s\in \Gamma(X,E)$, 
\begin{align*}
\nabla^\psi(fs)&=\text{id}_{\la{L}^\star}\otimes \big([\psi]^{-1}\circ \phi_E^{-1}\big)\circ \nabla\big(\phi_E^{-1}\circ\psi(fs)\big)\\
&=\text{id}_{\la{L}^\star}\otimes \big([\psi]^{-1}\circ \phi_E^{-1}\big)\circ \nabla\Big(f\big(\phi_E^{-1}\circ\psi\big)(s)\Big)\\
&=\text{id}_{\la{L}^\star}\otimes \big([\psi]^{-1}\circ \phi_E^{-1}\big) \Big[d_\la{L}(f)\otimes\psi(s)+\phi^\star(f)\nabla\big(\phi_E^{-1}\circ\psi(s)\big)\Big]\\
&=d_\la{L}(f)\phi_E(s)+\phi^\star(f)\nabla^\psi(s),
\end{align*}
also $\phi^\dagger_\la{L}\otimes \phi_E\circ \nabla^\psi=\nabla^\psi\circ \phi$ for $\psi\in \text{H-Gau}(E)$ implies the map $\odot$ is well defined;
\item the map $\odot$ is a left action of $\text{H-Gau}(E)$ on $A(E,\la{L})$.
\end{enumerate}

\begin{remark}{\cite[cf.Theorem 6]{SMMRF}}\rm{
The group $\text{H-Gau}(E)$ has a Lie group as well as a Hom-Lie group structure, denoted by $\left(\text{H-Gau}(E),\bigcdot,\phi_E,\text{Ad}_{\phi_E}\right)$, because 
$$
\text{Ad}_{\phi_E}\circ \bigcdot(x,y)=\bigcdot(x,y)\qquad\big(x,y\in \text{H-Gau}(E)\big)
$$
with Lie bracket structure on the Lie algebra $\text{H-}\mathfrak{g}\text{au}(E)=\mathcal{A}^0_\la{L}\big(X,\text{End}_{\phi_E}(E)\big)$ can be described using the Equation \eqref{eq.1} 
\begin{align*}
[x,y]&=x\bigcdot y-y\bigcdot x\\
&=\text{Ad}_{\phi_E}(x\bigcdot y-y\bigcdot x)\\
&=[x,y]_{\text{H-Gau}(E)}
\end{align*}
}
\end{remark}
For a given $\la{L}$-connection, $\nabla$ the isotropy subgroup $\text{H-Gau}_\nabla(E)\subset \text{H-Gau}(E)$ is given by
$$
\text{H-Gau}_{\nabla}(E)=\{\psi\in \text{H-Gau}(E)\mid \nabla^\psi=\nabla\},
$$
\begin{definition}\rm{
An $\la{L}$-connection $\nabla$ with isotropy subgroup $\text{H-Gau}_{\nabla}(E)=\mathbb{K}^\star\bigcdo \phi_E$, is called an irreducible $\la{L}$-connection, which will be denoted by $\widehat{A}(E,\la{L})$.
}
\end{definition}
It is easy to observe 
$$
\text{H-Gau}_{\nabla^\psi}(E)=[\psi]^{-1}\bigcdot \text{H-Gau}_\nabla(E)\bigcdot \psi
$$
Hence, the space of irreducible $\la{L}$-connections $\widehat{A}(E,\la{L})$ is closed under the action of $\text{H-Gau}(E)$ and we have the quotient spaces
\begin{align*}
p:A(E,\la{L})\rightarrow &\bigslant{A(E,\la{L})}{\text{H-Gau}(E)}=B(E,\la{L})\\
&\text{and}\\
\widehat{p}:\widehat{A}(E,\la{L})\rightarrow &\bigslant{\widehat{A}(E,\la{L})}{\text{H-Gau}(E)}=\widehat{B}(E,\la{L})
\end{align*}

\begin{remark}\rm{
For a given smooth Hom-bundle $E$ and smooth Hom-Lie algebroid $\la{L}$ on a smooth manifold $X$, define the reduced H-Gauge group as the quotient space $\text{H-Gau}^r(E)=\bigslant{\text{H-Gau}(E)}{\mathbb{K}^\star  \bigcdo \phi_E}$. Because for each $\la{L}$-connection $\nabla$, we have $\mathbb{K}^\star\bigcdo \phi_E\subset \text{H-Gau}_\nabla(E)$, the action of $\text{H-Gau}(E)$ reduces to the action of $\text{H-Gau}^r(E)$ and the moduli spaces, which we are interested in are the $\text{H-Gau}^r(E)$ equivalence classes
\begin{align*}
p:A(E,\la{L})\rightarrow &\bigslant{A(E,\la{L})}{\text{H-Gau}^r(E)}=B(E,\la{L})\\
&\text{and}\\
\widehat{p}:\widehat{A}(E,\la{L})\rightarrow &\bigslant{\widehat{A}(E,\la{L})}{\text{H-Gau}^r(E)}=\widehat{B}(E,\la{L})
\end{align*}
}
\end{remark}
\begin{remark}\label{connrem}\rm{
For a given Hom-Lie algebroid connection $\nabla\in A(E,\la{L})$ and $\psi\in \text{H-Gau}(E)$, the transformed connection $\nabla^\psi$ is given by
\begin{align*}
\nabla^\psi&=\text{id}_{\la{L}^\star}\otimes ([\psi^{-1}]\circ \phi_E^{-1})\circ \nabla_E\circ \phi_E^{-1}\circ \psi
\end{align*}
Using Remark \ref{remar}, we have
\begin{align}\label{equa1}
(\phi^\dagger_\la{L})^{-1}\otimes \phi_E^{-1}\circ(\nabla^{\text{End}(E)}\psi)\circ \phi_E&=\nabla^E\circ \phi_E^{-1}\circ \psi-\text{id}_{\la{L}^\star}\otimes (\psi^{-1}\otimes \phi_E^{-1})\circ\nabla^E\nonumber \\ 
\text{id}_{\la{L}^\star}\otimes ([\psi^{-1}]\circ \phi_E^{-1})\circ (\phi^\dagger_\la{L})^{-1}\otimes \phi_E^{-1}\circ(\nabla^{\text{End}(E)}\psi)\circ \phi_E&=\nabla^\psi-\nabla^E 
\end{align}
Using Remark \ref{remar1}, and applying $\nabla^{\text{End}(E)}$ on the equality $[\psi^{-1}]\circ \phi_E^{-1}\circ\psi=\phi_E$, we get
\begin{equation}\label{equa2}
(\phi^\dagger_\la{L})^{-1}\otimes \phi_E^{-1}\circ(\nabla^{\text{End}(E)}[\psi^{-1}])\circ \psi+\text{id}_{\la{L}^\star}\otimes([\psi^{-1}]\circ \phi_E^{-1})\circ(\phi^\dagger_\la{L})^{-1}\otimes \phi_E^{-1}\circ (\nabla^{\text{End}(E)}\psi)\circ \phi_E=0
\end{equation}
From Equations \eqref{equa1} and \eqref{equa2}, we have
\begin{equation}\label{equa3}
\nabla^\psi=\nabla-(\phi^\dagger_\la{L})^{-1}\otimes \phi_E^{-1}\circ(\nabla^{\text{End}(E)}[\psi^{-1}])\circ \psi
\end{equation}
More generally, for some $\la{L}$-connection $\nabla=\nabla_0+\alpha$, $\alpha\in \mathcal{A}^1_\la{L}\big(X,\text{End}(E)\big)$, we have
\begin{align*}
\nabla^\psi&=\text{id}_{\la{L}^\star}\otimes ([\psi^{-1}]\circ \phi_E^{-1})\circ(\nabla_0+\alpha)\circ \phi_E^{-1}\circ \psi\\
&=\text{id}_{\la{L}^\star}\otimes ([\psi^{-1}]\circ \phi_E^{-1})\circ \nabla_0\circ \phi_E^{-1}\circ \psi+\text{id}_{\la{L}^\star}\otimes ([\psi^{-1}]\circ \phi_E^{-1})\circ \alpha\circ \phi_E^{-1}\circ \psi\\
&=\nabla_0+\text{id}_{\la{L}^\star}\otimes ([\psi^{-1}]\circ \phi_E^{-1})\circ (\phi^\dagger_\la{L})^{-1}\otimes \phi_E^{-1}\circ (\nabla_0^{\text{End}(E)}\psi)\circ \phi_E+\\
&\qquad\quad\text{id}_{\la{L}^\star}\otimes ([\psi^{-1}]\circ \phi_E^{-1})\circ \alpha\circ \phi_E^{-1}\circ \psi\qquad\qquad(\text{from} \eqref{equa1})\\
&=\nabla_0+\text{id}_{\la{L}^\star}\otimes [\psi^{-1}]\bigcdot (\phi^\dagger_\la{L})^{-1}\otimes \phi_E^{-1}\circ (\nabla_0^{\text{End}(E)}\psi)\circ \phi_E+\text{id}_{\la{L}^\star}\otimes [\psi^{-1}]\bigcdot \alpha\bigcdot \psi
\end{align*}
And, we can write
\begin{align*}
\nabla^\psi&=\nabla_0+\alpha^\psi\text{ such that }\\
\alpha^\psi&=\text{id}_{\la{L}^\star}\otimes [\psi^{-1}]\bigcdot (\phi^\dagger_\la{L})^{-1}\otimes \phi_E^{-1}\circ (\nabla_0^{\text{End}(E)}\psi)\circ \phi_E+\text{id}_{\la{L}^\star}\otimes [\psi^{-1}]\bigcdot \alpha\bigcdot \psi
\end{align*}
}
\end{remark}
\section{Sobolve H-Gauge group action}\label{sec-4}
In this section, we have discussed Sobolev completion of $\text{H-Gau}(E)$-moduli spaces, using Sobolev theory of function spaces. The main theorems of this section are Theorems \ref{main1} and \ref{main2}, in which we have described the Hilbert manifold structure on the moduli space $\widehat{B}(E,\la{L})_l$ and $\widehat{p}:\widehat{A}(E,\la{L})_l\rightarrow \widehat{B}(E,\la{L})_l$ has principal-$\text{H-Gau}(E)_{l+1}^r$ bundle structure.

A complex (resp. real) Hom-bundle $E$ is said to have Hom-Hermitian (resp. Hom-Euclidean) metric if there is a Hermitian (resp. Euclidean) metric $(\cdot,\cdot)_{h^E}$ on the bundle $E$ such that 
$$
\big( \phi_E(s_1),\phi_E(s_2)\big)_{h^E}=\phi^\star(s_1,s_2)_{h^E}\quad\text{ for }s_1,s_2\in \Gamma(X,E)
$$

\subsection{Sobolev completion of space of $\la{L}$-connections}
Let $E,\la{L}$ be a complex (resp. real) Hom-bundle, Hom-Lie algebroid respectively over a compact manifold $X$ with Hom-Hermitian (resp. Hom-Euclidean) metric corresponding to Hermitian (resp. Euclidean) metrices $h^E, h^\la{L}$ respectively and Riemannian metric $g$ on the manifold $X$. There is an induced Hom-metric on the Hom-bundle $\Lambda^k\la{L}\otimes E$ $(0\leq k)$ associated to the Hermitian (resp. Euclidean) metric $h^E_\la{L}$ induced by $h^E$ and $h^\la{L}$. The Riemannian metric $g$ induces a volume form $\text{vol}(g)$, which further induces Borel measure $\mu_g$ on the smooth manifold $X$. 

Let $L^2_l\big(\mathcal{A}_\la{L}^0(X,E)\big)(l\in \mathbb{Z}^+)=\mathcal{A}_\la{L}^0(X,E)_l$ be the space of equivalence classes of Borel measurable sections with weak derivatives of order upto $l$ are square integrable. The space $\mathcal{A}_\la{L}^0(X,E)_l(l\in \mathbb{Z}^+)$ is called the Sobolev completion of space of global sections of bundle $E$, which is a Hilbert space with inner product
$$
\langle s_1,s_2\rangle_{L^2_l}=\langle s_1,s_2\rangle_{l}=\displaystyle\sum_{j=0}^l \langle \nabla^js_1,\nabla^js_2\rangle_{L^2}\qquad(s_1,s_2\in \mathcal{A}_\la{L}^0(X,E)_l)
$$

where $\langle \nabla^js_1,\nabla^js_2\rangle_{L^2}$ can be computed using metric $h^E$ and $h^E_\la{L}$. The metric $h^E$ induces a metric on the bundle $\text{End}(E)\cong \phi!E\otimes Hom(E,\phi!\underline{\mathbb{K}})$. 
Using this induced metric and the given Hermitian metric on $\la{L}$, we can describe the Sobolev completion of the space $\mathcal{A}_\la{L}^k\big(X,\text{End}(E)\big)$, which will be written as $L^2_l\Big(\mathcal{A}_\la{L}^k\big(X,\text{End}(E)\big)\Big)=\mathcal{A}_\la{L}^k\big(X,\text{End}(E)\big)_l$

Using the Lemma \ref{liealgconn}, define the Sobolev space of $\la{L}$-connections as
$$
A(E,\la{L})_l=\{\nabla_0+\alpha\mid \alpha\in \mathcal{A}^1_\la{L}\big(X,\text{End}_{\phi_E}(E)\big)_l\}
$$
where $\mathcal{A}^1_\la{L}\big(X,\text{End}_{\phi_E}(E)\big)_l=\{x\in \mathcal{A}^1_\la{L}\big(X,\text{End}(E)\big)_l\mid\phi^\dagger_\la{L}\otimes\phi_E\circ x=x\circ \phi^\dagger_\la{L}\otimes\phi_E\}$ and $\phi^\dagger_\la{L}\otimes\phi_E$ is continuous extension of the smooth map $\phi^\dagger_\la{L}\otimes\phi_E:\mathcal{A}^1_\la{L}(X,E)\rightarrow \mathcal{A}^1_\la{L}(X,E)$ to appropriate Sobolev spaces, using Sobolev multiplication theorem in the range $l>\frac{1}{2}\text{dim}_\mathbb{R}(X)$. 

The space $\mathcal{A}^1_\la{L}\big(X,\text{End}_{\phi_E}(E)\big)_l\subset \mathcal{A}^1_\la{L}\big(X,\text{End}(E)\big)_l$ is a closed subspace, furthermore a Hilbert space, which implies the Sobolev space of $\la{L}$-connections is a Hilbert manifold.
\subsection{Sobolev completion of space of H-gauge transformations}
For $l>\frac{1}{2}\text{dim}_\mathbb{R}X$, using Sobolev multiplication theorem, we have the continuously extended map
\begin{align*}
\mathcal{A}^0_\la{L}\big(X,\text{End}_{\phi_E}(E)\big)_{l+1}\times \mathcal{A}^0_\la{L}\big(X,\text{End}_{\phi_E}(E)\big)_{l+1}&\rightarrow \mathcal{A}^0_\la{L}\big(X,\text{End}_{\phi_E}(E)\big)_{l+1}\\\big( (\psi_1,\psi_2)&\mapsto \psi_1\bigcdot\psi_2=\psi_1\circ \phi_E^{-1}\circ \psi_2\big)
\end{align*}
which makes the space $\mathcal{A}^0_\la{L}\big(X,\text{End}_{\phi_E}(E)\big)_{l+1}$, a Banach $\mathbb{K}$-algebra.

Define the Sobolev space of H-gauge group
\begin{equation*}
  \text{H-Gau}(E)_{l+1} = \left\{ \psi\in \mathcal{A}^0_\la{L}\big(X,\text{End}_{\phi_E}(E)\big)_{l+1}\ \middle\vert \begin{array}{l}
    \psi\bigcdot \psi'=\psi'\bigcdot \psi=\phi_E \\
    \text{ for some }\psi'\in \mathcal{A}^0_\la{L}\big(X,\text{End}_{\phi_E}(E)\big)_{l+1}
  \end{array}\right\}
\end{equation*}
The Sobolev space of H-gauge group is the subspace of invertible elements in the Hilbert space $\mathcal{A}^0_\la{L}\big(X,\text{End}_{\phi_E}(E)\big)_{l+1}$, implies $\big(\text{H-Gau}(E)\big)_{l+1}$ is a Hilbert-Lie group. The Lie bracket structure on the Lie algebra  
$$
\text{H-}\mathfrak{gau}(E)_{l+1}=\mathcal{A}^0_\la{L}\big(X,\text{End}_{\phi_E}(E)\big)_{l+1}
$$
can be described by continuously extending the map $[\cdot,\cdot]_\la{L}$ described in Equation \eqref{eq.1} to appropriate Sobolev space, using Sobolev multiplication theorem.

\subsection{Sobolev H-gauge action on the Sobolev space of $\la{L}$-connection}
The $\text{H-Gau}(E)$ action on the space $A(E,\la{L})$ described in Equation \eqref{gaugeact}, can be extended continuously to the Sobolev completion spaces
\begin{align*}
\odot:\text{H-Gau}(E)_{l+1}\times A(E,\la{L})_l&\rightarrow A(E,\la{L})_l\\
(\psi,\nabla=\nabla_0+\alpha)&\mapsto\nabla^\psi=\nabla_0+\text{id}_{\la{L}^\star}\otimes([\psi^{-1}]\circ \phi_E^{-1})\circ\big((\phi^\dagger_\la{L})^{-1}\otimes \phi_E^{-1}\big)\circ(d^{\nabla_0}\psi)\circ \phi_E+\\
&\qquad\qquad\text{id}_{\la{L}^\star}\otimes(\psi\circ \phi_E^{-1})\circ\alpha\circ\phi_E^{-1}\circ\psi^{-1}\qquad(\text{ see Equation }\eqref{connrem})
\end{align*}
where $d^{\nabla_0}:\mathcal{A}^{\bigcdot}_\la{L}\big(X,\text{End}_{\phi_E}(E)\big)_l\rightarrow\mathcal{A}^{\bigcdot +1}_\la{L}\big(X,\text{End}_{\phi_E}(E)\big)_{l-1}$ is continous extension of the degree 1 operator $d^{\nabla_0}$ assciated to the connection $\nabla^{\text{End}(E)}_0$ (see Equation \eqref{operconn1}) and $\phi_E,\psi$ are continuously extended locally linear maps on finite dimensional Sobolev space; hence the Sobolev Gauge action $\odot$ is a smooth map.

Similar to the smooth case discussed in previous section, for a Sobolev $\la{L}$-connection $\nabla$, the isotropy subgroup $\text{H-Gau}_\nabla(E)_{l+1}\subset \text{H-Gau}(E)_{l+1}$ is given by
\begin{equation}\label{sobgaueq}
\text{H-Gau}_\nabla(E)_{l+1}=\{\psi\in \text{H-Gau}(E)_{l+1}\mid \nabla^\psi=\nabla\}
\end{equation}
and a Sobolev $\la{L}$-connection is said to be an irreducible $\la{L}$-connection if $\text{H-Gau}_\nabla(E)=\mathbb{K}^\star\cdot \phi_E$. The collection of irreducible Sobolev $\la{L}$-connections will be denoted by $\widehat{A}(E,\la{L})_l\subset A(E,\la{L})_l$. Similar to the smooth case, it can be verified that,
$$
\text{H-Gau}_{\nabla^\psi}(E)_{l+1}=[\psi]^{-1}\bigcdot\text{H-Gau}(E)_{l+1}\bigcdot\psi
$$
and the space $\widehat{A}(E,\la{L})_l$ is closed under $\text{H-gauge}$ action. Define the $\text{H-gauge}$ theoretic moduli spaces
\begin{align*}
p:A(E,\la{L})_l\rightarrow &\bigslant{A(E,\la{L})_l}{\text{H-Gau}(E)_{l+1}}=B(E,\la{L})_l\\
&\text{and}\\
\widehat{p}:\widehat{A}(E,\la{L})_l\rightarrow &\bigslant{\widehat{A}(E,\la{L})_l}{\text{H-Gau}(E)_{l+1}}=\widehat{B}(E,\la{L})_l
\end{align*}
For an irreducible Sobolev $\la{L}$-connection $\nabla\in A(E,\la{L})_l$, the isotropy subgroup $\text{H-Gau}_\nabla(E)_{l+1}\subset \text{H-Gau}(E)_{l+1}$ is $\mathbb{K}^\star\cdot\phi_E$, we define the reduced Sobolev H-gauge group as,
$$
\text{Gau}^r(E)_{l+1}=\bigslant{\text{Gau}(E)_{l+1}}{\mathbb{K}^\star\cdot\phi_E}
$$
The left action of $\text{Gau}^r(E)_{l+1}$ on the space $\widehat{A}(X,\la{L})_l$ is a free action.
\begin{theorem}\cite[cf. Theorem II.2]{HGKHN}\label{toprove}
For a given Banach Lie group $G$ over a field $\mathbb{K}$ with a normal Banach Lie subgroup $N$ with Lie algebras $\mathfrak{g},\mathfrak{n}$ respectively. Then $\bigslant{G}{N}$ is a Banach Lie group with Lie algebra $\bigslant{\mathfrak{g}}{\mathfrak{n}}$,  which can be described in unique way such that the map $q:G\rightarrow \bigslant{G}{N}$ is a smooth map. Moreover, for any Banach manifold $X$ a map $f:\bigslant{G}{N}\rightarrow X$ is smooth iff $f\circ q$ is smooth.
\end{theorem}
Using above theorem, a Lie group structure on the quotient space $\text{H-Gau}^r(E)$ can be described in a unique way with Lie algebra
$$
\text{H-}\mathfrak{gau}^r(E)_{l+1}=\bigslant{\mathcal{A}^0_\la{L}\big(X,\text{End}_{\phi_E}(E)\big)_{l+1}}{\mathbb{K}^\star\bigcdo \phi_E}=\mathcal{A}^0_\la{L}\big(X,\text{End}_{\phi_E}(E)\big)_{l+1}^0
$$
where the space $\mathcal{A}^0_\la{L}\big(X,\text{End}_{\phi_E}(E)\big)_{l+1}^0$ can be described using the $L^2$-orthogonal decomposition
$$
\mathcal{A}^0_\la{L}\big(X,\text{End}_{\phi_E}(E)\big)_{l+1}=\mathcal{A}^0_\la{L}\big(X,\text{End}_{\phi_E}(E)\big)_{l+1}^0\oplus\mathbb{K}^\star\bigcdo \phi_E
$$
The space $\mathcal{A}^0_\la{L}\big(X,\text{End}_{\phi_E}(E)\big)_{l+1}^0$ is given by
$$
\mathcal{A}^0_\la{L}\big(X,\text{End}_{\phi_E}(E)\big)_{l+1}^0=\left\{s\in \mathcal{A}^0_\la{L}\big(X,\text{End}_{\phi_E}(E)\big)_{l+1}\mid \int_X\text{tr}(\phi_E^\star\circ s)d\mu_g=0\right\}
$$
where $\phi_E^\star$ is adjoint operator w.r.t Hermitian metric $h^E$.
\begin{remark}\rm{
For a given $\alpha\in \mathcal{A}^1_\la{L}\big(X,\text{End}(E)\big)$, we have order 0, degree 1 differential operator
$$
\text{ad}(\alpha):\mathcal{A}_\la{L}^0\big(X,\text{End}(E)\big)\rightarrow \mathcal{A}^1_\la{L}\big(X,\text{End}(E)\big)
$$
given by $\text{ad}(\alpha)(\beta)=[\alpha,\beta]_\la{L}$.

The adjoint operator $\text{ad}(\alpha)^\star$ w.r.t the induced metric on the space $\mathcal{A}^0_\la{L}\big(X,\text{End}(E)\big)$ gives a ses-quilinear map
$$
m:\mathcal{A}^1_\la{L}\big(X,\text{End}(E)\big)\times \mathcal{A}^1_\la{L}\big(X,\text{End}(E)\big)\rightarrow \mathcal{A}^0_\la{L}\big(X,\text{End}(E)\big)\quad \big(m(\alpha,\beta)=\text{ad}(\alpha)^\star(\beta)\big)
$$
Using Sobolev multiplication theorem \big($l>\frac{1}{2}\text{dim}_\mathbb{R}(X)$\big) the map $m$ can be extended continuously to a map $m_l$,  given by
$$
m_{l}:\mathcal{A}^1_\la{L}\big(X,\text{End}(E)\big)_{l}\times \mathcal{A}^1_\la{L}\big(X,\text{End}(E)\big)_{l}\rightarrow \mathcal{A}^0_\la{L}\big(X,\text{End}(E)\big)_{l}\quad \big(m_l(\alpha,\beta)=\text{ad}(\alpha)^\star(\beta)\big)
$$
and for a fixed $\alpha\in \mathcal{A}^1_\la{L}\big(X,\text{End}(E)\big)_{l}$, we have the continuous $(\phi^\star)^{-1}$-linear map defined on appropriate Sobolev space
$$
\text{ad}(\alpha)^\star:\mathcal{A}^1_\la{L}\big(X,\text{End}(E)\big)_{l}\rightarrow \mathcal{A}^0_\la{L}\big(X,\text{End}(E)\big)_{l}
$$ 
Using Sobolev embedding theorem, we have the compact embeddings given by
$$
i_{l}:\mathcal{A}^0_\la{L}\big(X,\text{End}(E)\big)_{l}\rightarrow \mathcal{A}^0_\la{L}\big(X,\text{End}(E)\big)_{l-1}\quad(l\in \mathbb{Z})
$$
For a given Sobolev $\la{L}$-connection $\nabla=\nabla_0+\alpha$, we can write the differential operators
\begin{align*}
d^\nabla:&\mathcal{A}^0_\la{L}\big(X,\text{End}(E)\big)_{l+1}\rightarrow \mathcal{A}^1_\la{L}\big(X,\text{End}(E)\big)_{l}\quad \text{and}\\
(d^\nabla)^\star:& \mathcal{A}^1_\la{L}\big(X,\text{End}(E)\big)_{l}\rightarrow \mathcal{A}^1_\la{L}\big(X,\text{End}(E)\big)_{l-1}
\end{align*}
as, $d^\nabla=d^{\nabla_0}+\text{ad}(\alpha)\circ i_{l+1}$ and $(d^\nabla)^\star=(d^{\nabla_0})^\star+i_l\circ \text{ad}(\alpha)^\star$. The operator $(d^{\nabla_0})^\star$ is continuous extension to appropriate Sobolev spaces of the formal adjoint of $d^{\nabla_0}$ w.r.t the Hermitian metric $h^{\text{End}(E)}$.
}
\end{remark}
From now on we will assume, the given Hom-Lie algebroid $\la{L}$ is transitive, which means the anchor map $\mathfrak{a}_\la{L}:\la{L}\rightarrow \phi^!T_X$ is surjective.
\begin{lemma}
For any $\nabla\in A(E,\la{L})_l$, the composition of operators $d^\nabla$ and $(d^\nabla)^\star$
$$
(d^\nabla)^\star\circ d^\nabla:\mathcal{A}^0_\la{L}\big(X,\text{End}(E)\big)_{l+1}\rightarrow \mathcal{A}^0_\la{L}\big(X,\text{End}(E)\big)_{l+1}
$$
is a Fredholm operator, for all $l>\frac{1}{2}\text{dim}_\mathbb{R}(X)$.
\end{lemma}
\begin{proof}
It is enough to prove $(d^{\nabla_0})^\star\circ d^{\nabla_0}$ is a Fredholm operator, because
\begin{align*}
(d^\nabla)^\star\circ d^\nabla&=(d^{\nabla_0})^\star\circ d^{\nabla_0}+i\circ \text{ad}(\alpha)^\star\circ d^{\nabla_0}+i\circ \text{ad}(\alpha)^\star\circ \text{ad}(\alpha)\circ i+(d^{\nabla_0})^\star\circ \text{ad}(\alpha)\circ i
\end{align*}
and $i\circ \text{ad}(\alpha)^\star,\text{ad}(\alpha)\circ i$ are compact implies $i\circ \text{ad}(\alpha)^\star\circ d^{\nabla_0}+i\circ \text{ad}(\alpha)^\star\circ \text{ad}(\alpha)\circ i+(d^{\nabla_0})^\star\circ \text{ad}(\alpha)\circ i$ are compact operators. And sum of a Fredholm and a compact operator is a Fredholm operator, so it is enough to prove that $(d^{\nabla_0})^\star\circ d^{\nabla_0}$ is an elliptic operator, which is true in case the Hom-Lie algebroid $\la{L}$ is transitive because the principal symbol 
\begin{align*}
\sigma_1\big((d^{\nabla_0})^\star\circ d^{\nabla_0}\big)(\xi_x)&=\sigma_1\big((d^{\nabla_0})^\star\big)(\xi_x)\circ \sigma_1(d^{\nabla_0})(\xi_x)\qquad (\xi_x\in \phi^!T_{X,x}^\star\setminus\{0\})
\end{align*} 
is an isomorphism if $\sigma_1(d^{\nabla_0})(\xi_x)(\underline{\hspace{6pt}})=\mathfrak{a}_\la{L}^\star(\xi_x)\otimes\underline{\hspace{6pt}}$ is isomorphism, which is true, if $\mathfrak{a}_{\la{L},x}^\star$ is injective or, the anchor map $\mathfrak{a}_\la{L}$ is surjective.
\end{proof}
\begin{lemma}
For any Sobolev $\la{L}$-connection $\nabla\in A(E,\la{L})$, we have the following $L^2$-orthogonal decomposition
\begin{equation}\label{orthodecom}
\mathcal{A}^1_\la{L}\big(X,\emph{End}(E)\big)_l=\emph{im}(d^\nabla)\oplus \emph{ker}\big((d^\nabla)^\star\big)
\end{equation}
for all $l>\frac{1}{2}\text{dim}_\mathbb{R}X$.
\end{lemma}
\begin{proof}
Using Fredholmness of the operator $\Delta=(d^\nabla)^\star\circ d^\nabla$, the dimension of $\text{ker}(\Delta)\subset \mathcal{A}_\la{L}^0\big(X,\text{End}(E)\big)_{l+1}$ is finite and $\text{im}(\Delta)\subset \mathcal{A}_\la{L}^0\big(X,\text{End}(E)\big)_{l-1}$ is a closed subspace. Furthermore, we have the $L^2$-orthogonal decomposition of $\mathcal{A}^0_\la{L}\big(X,\text{End}(E)\big)_{l+1}$ in closed subspaces
$$
\mathcal{A}^0_\la{L}\big(X,\text{End}(E)\big)_{l+1}=\text{ker}(\Delta)+\text{ker}(\Delta)^\perp
$$
and using the fact that $\text{im}(\Delta)\subset \mathcal{A}_\la{L}^0\big(X,\text{End}(E)\big)_{l-1}$ is a closed subspace, the bijective map
$$
{\Delta}_{\mkern 1mu \vrule height 2ex\mkern2mu \text{ker}(\Delta)^\perp}:\text{ker}(\Delta)^\perp\rightarrow \text{im}(\Delta)
$$
is a continuous map. Using Banach open mapping theorem $\big({\Delta}_{\mkern 1mu \vrule height 2ex\mkern2mu \text{ker}(\Delta)^\perp}\big)^{-1}$ is a continuous map, furthermore $\text{id}_Y-d^\nabla\circ \big({\Delta}_{\mkern 1mu \vrule height 2ex\mkern2mu \text{ker}(\Delta)^\perp}\big)^{-1}\circ (d^\nabla)^\star$ is a continuous map on $Y=\big((d^\nabla)^\star\big)^{-1}\big(\text{im}(\Delta)\big)$. Note that $\text{ker}(\Delta)=\text{ker}(d^\nabla)$ implies $\text{im}(d^\nabla)= \text{im}(d^\nabla)_{\mkern 1mu \vrule height 2ex\mkern2mu \text{ker}(\Delta)^\perp}$ and 
$$
\text{im}(d^\nabla)=\text{ker}\big(\text{id}_Y-d^\nabla\circ \big({\Delta}_{\mkern 1mu \vrule height 2ex\mkern2mu \text{ker}(\Delta)^\perp}\big)^{-1}\circ (d^\nabla)^\star\big)
$$
is a closed subspace of $\mathcal{A}^1_\la{L}\big(X,\text{End}(E)\big)_l$.

Furthermore, we have the $L^2$-orthogonal decomposition 
$$
\mathcal{A}^1_\la{L}\big(X,\text{End}(E)\big)_l=\text{im}(d^\nabla)\oplus \text{im}(d^\nabla)^\perp
$$
Using the property of adjoint operator, we have $\text{im}(d^\nabla)^\perp=\text{ker}(d^\nabla)^\star$ and we have the required $L^2$-orthogonal decomposition.
\end{proof}
\begin{lemma}\label{remark}
Let $\nabla$ be an $\la{L}$-connection for a Hom-bundle $E$ over a compact manifold $X$. The following statements are equivalent
\begin{enumerate}
\item $\emph{H-Gau}_\nabla(E)=\mathbb{K}^\star\bigcdo \phi_E$
\item $\emph{ker}(\nabla^{\emph{End}(E)})=\mathbb{K}\bigcdo \phi_E$
\end{enumerate}
\end{lemma}
\begin{proof}
The proof for $(1)\implies (2)$ can be proved as follows.

Let $\psi\in \mathbb{K}^\star\bigcdo \phi_E$, we have $\nabla^\psi=\nabla$. Form the Equation \eqref{equa3}, we have $\psi\in \text{ker}(\nabla^{\text{End}(E)})$, which implies
\begin{equation}\label{eq}
\mathbb{K}\bigcdo \phi_E\subset \text{ker}(\nabla^{\text{End}(E)})
\end{equation}
Now let $\psi\in \text{ker}(\nabla^{\text{End}(E)})$ be a non-trivial element. Using compactness of manifold $X$, we can have $c\in \mathbb{K}$ with $|c|$ large enough such that $\phi_E\circ \Big(\text{id}_E+\bigslant{(\phi_E^{-1}\circ \psi)}{c}\Big)\in \text{H-Gau}(E)$. Also $\nabla^{\text{End}(E)}\biggl(\phi_E\circ \Big(\text{id}_E+\bigslant{(\phi_E^{-1}\circ \psi)}{c}\Big)\biggr)=0$ and using the Equation \eqref{equa3}, we have $\phi_E\circ \Big(\text{id}_E+\bigslant{(\phi_E^{-1}\circ \psi)}{c}\Big)\in \text{H-Gau}_\nabla(E)$. Assuming $(1)$ is true, we have $\psi\in \mathbb{K}\bigcdo \phi_E$ or,
\begin{equation}\label{eq2}
\text{ker}(\nabla^{\text{End}(E)})\subset \mathbb{K}\bigcdo \phi_E
\end{equation}
From Equations \eqref{eq} and \eqref{eq2}, we have $\text{ker}(\nabla^{\text{End}(E)})= \mathbb{K}\bigcdo \phi_E$.

To prove $(2)\implies (1)$, we need to show only $\text{H-Gau}(E)\subset K^\star\bigcdo \phi_E$ and can be proved using the Equation \eqref{equa3}.
\end{proof}
\begin{remark}\label{remark1}\rm{
The above theorem can be proved in the same line of arguements for Sobolev space of $\la{L}$-connections for Hom-bundle $E$ on compact manifold $X$ and we have the following lemma.
}
\end{remark}
\begin{lemma}\label{rema}
Let $\nabla$ be a Sobolev $\la{L}$-connection for a Hom-bundle $E$ over a compact manifold $X$. The following statements are equivalent
\begin{enumerate}
\item $\emph{H-Gau}_\nabla(E)_{l+1}=\mathbb{K}^\star\bigcdo \phi_E$
\item $\emph{ker}(\nabla^{\emph{End}(E)})=\mathbb{K}\bigcdo \phi_E$
\end{enumerate}
where $\phi_E$ is continuous extension of the map $\phi_E:\Gamma(X,E)\rightarrow \Gamma(X,E)$ to appropriate Sobolev spaces, using Sobolev multiplication theorem in the range $l>\frac{1}{2}\text{dim}_\mathbb{R}X$.
\end{lemma}
\begin{lemma}\label{openlconn}
The space of irreducible Sobolev $\la{L}$-connections $\widehat{A}(E,\la{L})_l\subset A(E,\la{L})_l$ is an open subspace for all $l>\frac{1}{2}\text{dim}_\mathbb{R}(X)$.
\end{lemma}
\begin{proof}
For a given Sobolev $\la{L}$-connection $\nabla=\nabla_0+\alpha$, we have a Fredholm operator $\Delta_\alpha$, with finite dimensional kernel space $\text{ker}(\Delta_\alpha)$; furthermore the composition of maps
$$
A(E,\la{L})_l\xrightarrow{\nabla\mapsto \Delta_\alpha}\mathcal{F}(\mathcal{A}^0_{\la{L},l+1},\mathcal{A}^0_{\la{L},l-1})\xrightarrow{\Delta_\alpha\mapsto \text{dim}(\text{ker}\Delta_\alpha)}\mathbb{R}
$$
is upper semi-continuous \cite{JCADSD}. For an irreducible Sobolev $\la{L}$-connection $\nabla=\nabla_0+\alpha$, $\text{dim}\big(\text{ker}(\Delta_\alpha)\big)=1$ (see Lemma \ref{rema}) also we know $\text{ker}(\Delta_\alpha)=\text{ker}(d_\alpha)$ implies the space of irreducible Sobolev $\la{L}$-connections $\widehat{A}(E,\la{L})_l\subset A(E,\la{L})_l$ is an open subspace, using semi-continuity. 
\end{proof}
Since space $\widehat{A}(E,\la{L})_l$ (resp. $A(E,\la{L})_l$) has image $\widehat{B}(E,\la{L})_l$ (resp. $B(E,\la{L})_l$) under left $\text{H-Gau}(E)_{l+1}$-action and using the Lemma \ref{openlconn}, the space $\widehat{B}(E,\la{L})_l\subset B(E,\la{L})_l$, is an open subspace under quotient toplogy. 

For $\nabla\in A(E,\la{L})_l$ and $\epsilon>0$, the Hilbert submanifold $\mathcal{O}_{\nabla,\epsilon}\subset A(E,\la{L})_l$
$$
\mathcal{O}_{\nabla,\epsilon}=\{\nabla+\alpha\mid \alpha\in \mathcal{A}^1_\la{L}\big(X,\text{End}_{\phi_E}(E)\big)_l, (d^\nabla)^\star(\alpha)=0\text{ and }\|\alpha\|_{L^2_l}<\epsilon\}
$$
has the tangent space 
$$
T_{\mathcal{O}_{\nabla,\epsilon}}=\text{Ker}\big((d^\nabla)^\star\big)\subset\mathcal{A}^1_\la{L}\big(X,\text{End}_{\phi_E}(E)\big)
$$
\subsection{Main theorem}
Considering a Hilbert submanifold $\mathcal{O}_{\nabla,\epsilon}\subset \widehat{A}(E,\la{L})_l\subset A(E,\la{L})_l$ for $\epsilon$ small enough, 
\begin{theorem}
The map
$$
\Psi_\nabla:\text{H-Gau}(E)^r_{l+1}\times \mathcal{O}_{\nabla,\epsilon}\rightarrow \widehat{A}(E,\la{L})_l
$$
given by
$$
\Psi_\nabla(\psi,\nabla+\alpha)=\psi\odot (\nabla+\alpha) 
$$
is smooth and a local diffeomorphism.
\end{theorem}
\begin{proof}
The differential map 
$$
d(\Psi_\nabla)_{(\phi_E,\nabla)}:\left[\mathcal{A}^0_\la{L}\big(X,\text{End}_{\phi_E}(E)\big)_l\right]^0\oplus\text{ker}\big((d^\nabla)^\star\big)\rightarrow \mathcal{A}^1_\la{L}\big(X,\text{End}_{\phi_E}(E)\big)_l
$$
given by
$$
d(\Psi_\nabla)_{(\phi_E,\nabla)}(A,B)=d^\nabla(A)+B
$$
Using the $L^2$-orthogonal decomposition Equation \eqref{orthodecom}, the differential map $d(\Psi_\nabla)_{(\phi_E,\nabla)}$ is surjective, it is injective as $\nabla$ is an irreducible $\la{L}$-connection. Using Banach open mapping theorem, the differential $d(\Psi_\nabla)_{(\phi_E,\nabla)}$ is isomorphism and using inverse function theorem, the map $\Psi_\nabla$ is local diffeomorphism around the point $(\phi_E,\nabla)\in \text{H-Gau}(E)^r_{l+1}\times \mathcal{O}_{\nabla,\epsilon}$.
\end{proof}
\begin{theorem}
For $\epsilon>0$ small enough, the map
$$
\widehat{p}_{\nabla,\epsilon}=\widehat{p}_{\mkern 1mu \vrule height 2ex\mkern2mu \mathcal{O}_{\nabla,\epsilon}}:\mathcal{O}_{\nabla,\epsilon}\rightarrow \widehat{B}(E,\la{L})
$$ 
is injective. If $\mathcal{U}_{\nabla,\epsilon}\subset \widehat{B}(E,\la{L})$ be the image, then the map
$$
\widehat{p}_{\nabla,\epsilon}:\mathcal{O}_{\nabla,\epsilon}\rightarrow \mathcal{U}_{\nabla,\epsilon}
$$
is homeomorphism.
\end{theorem}
\begin{proof}
For any two $\la{L}$-connections in $\mathcal{O}_{\nabla,\epsilon}$, say $\nabla+\alpha_1,\nabla+\alpha_2$, where $\alpha_i(i=1,2)\in \mathcal{A}^1_\la{L}\big(X,\text{End}_{\phi_E}(E)\big)_l$ such that there is a $\psi\in \mathcal{A}^0_\la{L}\big(X,\text{End}_{\phi_E}(E)\big)_{l+1}$ satisfying
\begin{equation}\label{eq1}
\nabla+\alpha_2=\psi\odot (\nabla+\alpha_2)
\end{equation}
then we need to show $\alpha_1=\alpha_2$.

Using the $L^2$-orthogonal decomposition
$$
\mathcal{A}^0_\la{L}\big(X,\text{End}(E)\big)_{l+1}=\text{ker}(d^\nabla)\oplus \text{ker}(d^\nabla)^\perp
$$
and the lemma \ref{rema}, let $\psi=c\phi_E+\psi_0$, where $\psi_0\in \text{ker}(d^\nabla)^\perp\cap \mathcal{A}^0_\la{L}\big(X,\text{End}_{\phi_E}(E)\big)_{l+1}$ and $c\in K$. Using Banach open mapping theorem, the map
$$
d^\nabla:\text{ker}(d^\nabla)^\perp\rightarrow \text{im}(d^\nabla)
$$
is an isomorphism between Hilbert spaces; hence a lower bounded operator. For $\psi_0\in \text{Gau}(E)_{l+1}$, let $\rho\in \mathbb{R}$ be a positive real number such that
$$
\rho\|\psi_0\|_{L^2_{l+1}}\leq \|d^\nabla(\psi_0)\|_{L^2_l}=\|d^\nabla(\psi)\|_{L^2_l}=\|(\phi^\dagger_\la{L})^{-1}\otimes \phi_E^{-1}\circ d^\nabla(\psi)\circ \phi_E\|_{L^2_l}\quad\big(\text{by def. of metric}\big)
$$
Using Remark \ref{connrem}, for the Equation \eqref{eq1}, we have
$$
\rho\|\psi_0\|_{L^2_{l+1}}\leq \|d^\nabla(\psi)\|_{L^2_l}=\|\text{id}_{\la{L}^\star}\otimes(\psi\circ \phi_E^{-1})\circ \alpha_2-\alpha_1\circ \phi_E^{-1}\circ \psi\|_{L^2_l}
$$
Using triangle inequality
\begin{align*}
\rho\|\psi_0\|_{L^2_{l+1}}&\leq \|\text{id}_{\la{L}^\star}\otimes(\psi\circ \phi_E^{-1})\circ \alpha_2\|_{L^2_l}+\|\alpha_1\circ \phi_E^{-1}\circ \psi\|_{L^2_l}\\
&\leq \rho_2\|\phi_E^{-1}\|_{L^2_{l+1}}\|\psi\|_{L^2_{l+1}}\|\alpha_2\|_{L^2_l}+\rho_1\|\phi_E^{-1}\|_{L^2_{l+1}}\|\psi\|_{L^2_{l+1}}\|\alpha_1\|_{L^2_l}\\
&\leq 2\tilde{\rho}\epsilon \|\phi_E^{-1}\|_{L^2_{l+1}}\|\psi\|_{L^2_{l+1}}\quad \big(\tilde{\rho}=\text{max}.(\rho_1,\rho_2);\|\alpha_1\|,\|\alpha_2\|<\epsilon\big)\\
&\leq 2\tilde{\rho}\epsilon \|\phi_E^{-1}\|_{L^2_{l+1}}\big(|c|\|\phi_E\|_{L^2_{l+1}}+\|\psi_0\|_{L^2_{l+1}}\big)
\end{align*}
we have
\begin{align*}
\|\psi_0\|_{L^2_{l+1}}&\leq\frac{2|c|\tilde{\rho}\epsilon\|\phi_E\|_{L^2_{l+1}}\|\phi_E^{-1}\|_{L^2_{l+1}}}{\rho-2\tilde{\rho}\epsilon\|\phi_E^{-1}\|_{L^2_{l+1}}}
\end{align*}
for $\epsilon<\frac{\rho}{2\tilde{\rho}\|\phi_E^{-1}\|_{L^2_{l+1}}}$. Note that $c\neq 0$ (otherwise $\psi_0=0$ implies $\psi$ is trivial) and 
$$
\|c^{-1}\psi-\phi_E\|=\frac{1}{|c|}\|\psi_0\|_{L^2_{l+1}}\leq\frac{2\tilde{\rho}\epsilon\|\phi_E\|_{L^2_{l+1}}\|\phi_E^{-1}\|_{L^2_{l+1}}}{\rho-2\tilde{\rho}\epsilon\|\phi_E^{-1}\|_{L^2_{l+1}}}
$$
For $\epsilon>0$ small enough, we can assume $\psi$ is near $\phi_E$ in $\text{H-Gau}(E)^r_{l+1}$, using injectivity of $\Psi_\nabla$ in above theorem, we have $\alpha_1=\alpha_2$.

Hence the map, $\widehat{p}_{\nabla,\epsilon}:\mathcal{O}_{\nabla,\epsilon}\rightarrow \mathcal{U}_{\nabla,\epsilon}$ is bijective continuous map, to prove homeomorphism we need to check if $(\widehat{p}_{\nabla,\epsilon})$ is an open map, which is true because action of $\text{H-Gau}(E)_{l+1}$ on $\widehat{A}(E,\la{L})_{l}$ is continuous and for any open subset $U\subset \mathcal{O}_{\nabla,\epsilon}$, we have
$$
(\widehat{p}_{\nabla,\epsilon})^{-1}\big((\widehat{p}_{\nabla,\epsilon})(U)\big)=\bigcupdot_{g\in \text{H-Gau}(E)_{l+1}} g\odot U
$$
which is open in $\widehat{A}(E,\la{L})_l$, implies $(\widehat{p}_{\nabla,\epsilon})(U)$ is open in $\mathcal{U}_{\nabla,\epsilon}$ for a given open subset $U\in \mathcal{O}_{\nabla,\epsilon}$.
\end{proof}
\begin{theorem}[Main theorem 1]\label{main1}
The map
\begin{align*}
\Psi_\nabla:\emph{H-Gau}(E)^r_{l+1}\times \mathcal{O}_{\nabla,\epsilon}&\rightarrow (\widehat{p})^{-1}(\mathcal{U}_{\nabla,\epsilon})\subset \widehat{A}(E,\la{L})_l\\
(g,\nabla)&\mapsto g\odot \nabla
\end{align*}
is diffeomorphism. In other words, the fibre bundle $\widehat{p}:\widehat{A}(E,\la{L})\rightarrow \widehat{B}(A,\la{L})_l$ is a principal ${\emph{H-Gau}(E)}^r_{l+1}$-bundle with trivialization $\{(\mathcal{U}_{\nabla_l,\epsilon_l},\Psi_{\nabla_l})\}_{l\in I}$.
\end{theorem}
\begin{proof}
Note that the image of $\mathcal{O}_{\nabla,\epsilon}$ under $\widehat{p}_{\nabla,\epsilon}$ is the $\text{H-Gau}^r_{l+1}(E)$-equivalence classes of connections in $\mathcal{O}_{\nabla,\epsilon}$ implies $\Psi_\nabla$ is surjective.

The action of $\text{H-Gau}^r_{l+1}(E)$ on $\widehat{A}(E,\la{L})$ is a free action and using injectivity of $\widehat{p}_{\nabla,\epsilon}$, it is easy to verify $\Psi_\nabla$ is injective, hence bijective.

The map 
$$
\Psi_\nabla:\emph{H-Gau}(E)^r_{l+1}\times \mathcal{O}_{\nabla,\epsilon}\rightarrow (\widehat{p})^{-1}(\mathcal{U}_{\nabla,\epsilon})
$$
is local diffeomorphism and bijective implies, it is diffeomorphism.
\end{proof}
\begin{theorem}[Main theorem 2]\label{main2}
The space $\widehat{B}(E,\la{L})_l$ is a smooth Hilbert manifold.
\end{theorem}
\begin{proof}
Consider the composition of maps
$$
g_\nabla:(\widehat{p})^{-1}(\mathcal{U}_{\nabla,\epsilon})\xrightarrow{\psi_\nabla}\text{H-Gau}(E)^r_{l+1}(X)\times \mathcal{O}_{\nabla,\epsilon}\xrightarrow{pr_1}\text{H-Gau}(E)^r_{l+1}
$$
where $pr_1:\text{H-Gau}(E)^r_{l+1}(X)\times \mathcal{O}_{\nabla,\epsilon}\rightarrow \text{H-Gau}(E)^r_{l+1}(X)$ is first projection map. There is a canonical co-ordinate chart $(\mathcal{U}_{\nabla,\epsilon},\sigma_\nabla)$ around $\nabla\in \widehat{B}(E,\la{L})_l$ with co-ordinate map
$$
\sigma_\nabla:\mathcal{U}_{\nabla,\epsilon}\rightarrow \mathcal{O}_{\nabla,\epsilon}\qquad(\sigma_\nabla(\nabla')=[g_\nabla(\nabla')]^{-1}\odot \nabla'
$$
It is easy to verify that $\sigma_\nabla\circ \widehat{p}_{\nabla,\epsilon}=\text{id}_{\mathcal{O}_{\nabla,\epsilon}}$ and $\widehat{p}_{\nabla,\epsilon}\circ \sigma_\nabla=\text{id}_{\mathcal{U}_{\nabla,\epsilon}}$.

Take two charts $(\mathcal{U}_{\nabla_l,\epsilon_l},\sigma_{\nabla_l})$ $(l=1,2)$ and $\nabla=\nabla_1+\alpha_1\in \mathcal{U}_{\nabla_1,\epsilon_1}$ such that $\widehat{p}_{\nabla_1,\epsilon_1}(\nabla)\in \mathcal{U}_{\nabla_1,\epsilon_1}\cap\mathcal{U}_{\nabla_2,\epsilon_2}$, we have
$$
\sigma_{\nabla_2}\circ \sigma_{\nabla_1}^{-1}(\nabla_1+\alpha_1)=\sigma_{\nabla_2}\big(\widehat{p}_{\nabla_1,\epsilon_1}(\nabla_1+\alpha_1)\big)=[g_{\nabla_2}(\nabla_1+\alpha_1)]^{-1}\odot (\nabla_1+\alpha_1)
$$
Since the Sobolev gauge action is a smooth map, the transition map is smooth and the space $\widehat{B}(E,\la{L})_l$ has smooth Hilbert manifold structure.
\end{proof}

\end{document}